\documentclass[12pt, a4paper]{article}

\usepackage[a4paper, margin=3cm]{geometry}
\usepackage{amsmath, amssymb, amsthm, graphicx}

\allowdisplaybreaks[4]

\newtheorem{lem}{Lemma}[section]
\newtheorem{thm}[lem]{Theorem}
\newtheorem{cor}[lem]{Corollary}

\numberwithin{equation}{section}
\numberwithin{figure}{section}

\renewcommand{\leq}{\leqslant}
\renewcommand{\geq}{\geqslant}

\newcommand{\bA}{\mathbf{A}}
\newcommand{\bc}{\mathbf{c}}
\newcommand{\be}{\mathbf{e}}
\newcommand{\bh}{\mathbf{h}}
\newcommand{\bp}{\mathbf{p}}
\newcommand{\bv}{\mathbf{v}}
\newcommand{\bw}{\mathbf{w}}
\newcommand{\bM}{\mathbf{M}}
\newcommand{\dimB}{\dim_{\rm B}}
\newcommand{\dimH}{\dim_{\rm H}}
\def\dist{{\rm dist}}
\def\supp{{\rm supp}}

\title{Attractors of directed graph IFSs that are not standard IFS attractors and their Hausdorff measure}     
\author{
G. C. Boore and K. J. Falconer \\
\footnotesize \emph{Mathematical Institute, University of St~Andrews,} \\
\footnotesize \emph{North Haugh, St~Andrews, Fife, KY16~9SS, Scotland.}
}
\begin{document}

\maketitle

\begin{abstract}
For directed graph iterated function systems (IFSs) defined on $\mathbb{R}$, we prove that a class of $2$-vertex directed graph IFSs have attractors that cannot be the attractors of standard ($1$-vertex directed graph) IFSs, with or without separation conditions. We also calculate their exact Hausdorff measure. Thus we are able to identify a new class of attractors for which the exact Hausdorff measure is known. 
\end{abstract}

\section{Introduction}\label{one}
The work of this paper was originally motivated by asking the question, ``Do we really get anything new with a directed graph IFS as opposed to a standard IFS?'' A standard IFS can always be represented as a $1$-vertex directed graph IFS so the question is really, ``Do we get anything new with a directed graph IFS with more than $1$ vertex as opposed to a $1$-vertex directed graph IFS?''. By restricting the systems under consideration to those defined on $\mathbb{R}$, we answer this question in the affirmative by proving that a class of $2$-vertex directed graph IFSs have attractors that cannot be the attractors of standard ($1$-vertex directed graph) IFSs, with or without separation conditions, overlapping or otherwise. We are also able to calculate the Hausdorff measure of these attractors and so we extend the class of attractors for which the exact Hausdorff measure is known.  

In what follows we will often write \emph{$k$-vertex IFS} as a shortening of $k$-vertex directed graph IFS.

We start, in Section \ref{three} by proving a general density result, Corollary \ref{2GcorG}, for directed graph IFSs defined on $\mathbb{R}^n$ for which the open set condition holds. In Section \ref{four}, Theorem \ref{2GthmN}, we give sufficient conditions for the calculation of the Hausdorff measure of both of the attractors of a class of $2$-vertex IFSs defined on $\mathbb{R}$. This adds to the work of Ayer and Strichartz \cite{Paper_Ayer_Strichartz} and Marion  \cite{Paper_Marion}. Then in Section \ref{five} we define the set of gap lengths of an attractor of any directed graph IFS defined on $\mathbb{R}$ for which the convex strong separation condition (CSSC) holds. In Section \ref{six}, by using sets of gap lengths to distinguish between attractors, we are able to show that a large family of directed graph IFSs, with any number of vertices, have attractors which are not attractors of standard ($1$-vertex) IFSs for which the CSSC holds, see Corollaries \ref{corC}, \ref{corCb} and Theorem \ref{thmA}. Finally in Section \ref{seven} we combine the results of Sections \ref{four} and \ref{six} to prove, in Theorems \ref{2GthmU} and \ref{2GthmV}, the existence of a class of $2$-vertex IFSs that have attractors that cannot be the attractors of standard ($1$-vertex) IFSs, with or without separation conditions. 

The attractors of these $2$-vertex IFSs are of interest not only because we are able to compute their Hausdorff measure, but also because they give us information about properties not shared by $1$-vertex IFSs. Also, because they are the attractors of such simple $2$-vertex IFSs, it seems likely that most directed graph IFSs produce genuinely new fractals, with many $3$-vertex IFSs having attractors that cannot be the attractors of $1$ or $2$-vertex IFSs and so on. 

A number of proofs involve checks that are routine or repetitive and so in these situations only a sketch or sample cases may be given, however full details of all proofs may be found in the thesis \cite{phdthesis_Boore}.  

\section{Notation and background theory}\label{two}
A \emph{directed graph}, $\bigl(V,E^*,i,t\bigr)$, consists of the set of all vertices $V$ and the set of all finite (directed) paths $E^*$, together with the initial and terminal vertex functions $i:E^* \to V$ and $t:E^* \to V$. $E^1$ denotes the set of all (directed) edges in the graph, that is the set of all paths of length one, with $E^1\subset E^*$. $V$ and $E^1$ are always assumed to be finite sets. We write $E^k$ for the set of all paths of length $k$, $E_{u}^k$ for the set of all paths of length $k$ starting at the vertex $u$, $E_{uv}^k$ for the set of all paths of length $k$ starting at the vertex $u$ and finishing at $v$ and so on. The \emph{initial} and \emph{terminal vertex} functions are defined as follows. Let $\be \in E^*$ be any finite path, then we may write $\be=e_1\cdots e_k$ for some edges $e_i\in E^1$, $1 \leq i \leq k$. The initial vertex of $\be$ is the initial vertex of its first edge, so $i(\be ) = i(e_1)$ and similarly for the terminal vertex $t(\be ) = t(e_k)$. 

We will often use a notation of the form $(A_c)_{c\in B}$ and $(A)_{c\in B}$, when $B$ is a finite set of $n$ elements, as this is just a convenient way of writing down ordered $n$-tuples. That is, if $B$ is ordered as $B=(b_1,b_2,\ldots,b_n)$, then $(A_c)_{c\in B}$ and $(A)_{c\in B}$ are the ordered $n$-tuples $(A_c)_{c\in B}=(A_{b_1},A_{b_2},\ldots,A_{b_n})$ and $(A)_{c\in B}=(A,A,\ldots,A)$.
  
We use the notation $\bigl(V,E^*,i,t,r,((C_{v},d_{v}))_{v \in V},(S_e)_{e \in E^1}\bigr)$ to indicate a \emph{directed graph IFS} and  $\bigl(V,E^*,i,t,r,p,((C_{v},d_{v}))_{v \in V},(S_e)_{e \in E^1}\bigr)$ for a \emph{directed graph IFS with probabilities}. $\bigl(V,E^*,i,t\bigr)$ is the directed graph of any such IFS and we always assume the directed graph  is \emph{strongly connected}, so there is at least one path connecting any two vertices. We also assume that each vertex in the directed graph has at least two edges leaving it, this is to avoid self-similar sets that consist of just single point sets, and attractors that are just scalar copies of those at other vertices, see \cite{Paper_Edgar_Mauldin}. The functions $r:E^*\to (0,1)$ and $p:E^*\to (0,1)$ assign contraction ratios and probabilities to the finite paths in the graph. To each vertex $v \in V$, is associated a complete metric space $(C_{v},d_{v})$ and to each directed edge $e\in E^1$ is assigned a contraction $S_{e}:C_{t(e)} \to C_{i(e)}$ which has the contraction ratio given by the function $r(e) = r_{e}$. We follow the convention already established in the literature, see \cite{Book_Edgar2} or \cite{Paper_Edgar_Mauldin}, that $S_e$ maps in the opposite direction to the direction of the edge it is associated with in the graph. 
 
The \emph{probability function} $p:E^*\to (0,1)$, where for an edge $e \in E^1$ we write $p(e)=p_{e}$, is such that $\sum_{ e\in E_{u}^1 }p_{e} = 1$, for any vertex $u \in V$. That is the probability weights across all the edges leaving a vertex always sum to one. For a path  $\be=e_1e_2\cdots e_k \in E^*$ we define $p(\be)=p_{\be}=p_{e_1}p_{e_2}\cdots p_{e_k}$. Similarly for the \emph{contraction ratio function} $r:E^*\to (0,1)$, the contraction ratio along a path $\be=e_1e_2\cdots e_k \in E^*$ is defined as $r(\be)=r_{\be}=r_{e_1}r_{e_2}\cdots r_{e_k}$. The ratio $r_{\be}$ is the ratio for the contraction $S_{\be}:C_{t(\be)} \to C_{i(\be)}$ along the path $\be$, where $S_{\be}=S_{e_1}\circ S_{e_2}\circ \cdots \circ S_{e_k}$.

In this paper we are only going to be concerned with directed graph IFSs defined on $n$-dimensional Euclidean space, with $((C_{v},d_{v}))_{v \in V}=((\mathbb{R}^n,\vert \ \ \vert))_{v \in V}$, and where $(S_e)_{e \in E^1}$ are contracting similarities and not just contractions. For any such IFS, $\bigl(V,E^*,i,t,r,((\mathbb{R}^n,\vert \ \ \vert))_{v \in V},(S_e)_{e \in E^1}\bigr)$, there exists a unique list of non-empty compact sets $(F_u)_{u \in V}$ satisfying  
\begin{equation}
\label{Theorem 1 b}
(F_u)_{u \in V}= \biggl( \ \bigcup_{e\in E_u^1} S_e(F_{t(e)}) \ \biggr)_{u \in V},
\end{equation} 
see Theorem 4.35, \cite{Book_Edgar2}. For the $1$-vertex case see Theorem 9.1, \cite{Book_KJF2}.  

We use the notation $\#V$ for the number of vertices in the set $V$, so $(\mathbb{R}^n)^{\#V}$ is the $\#V$-fold Cartesian product of $\mathbb{R}^n$. Also we write $K(\mathbb{R}^n)$ for the set of all non-empty compact subsets of $\mathbb{R}^n$ and $(K(\mathbb{R}^n))^{\#V}$ is the $\#V$-fold Cartesian product with $(F_u)_{u \in V} \in (K(\mathbb{R}^n))^{\#V}$.

For an IFS with probabilities, $\bigl(V,E^*,i,t,r,p,((\mathbb{R}^n,\vert \ \ \vert))_{v \in V},(S_e)_{e \in E^1}\bigr)$, there exists a unique list of Borel probability measures, $(\mu_u)_{u \in V}$, such that  
\begin{equation}
\label{Theorem 3 a}
(\mu_u(A_u))_{u \in V}=\biggl( \ \sum_{ \substack{ e\in E_u^1}  }p_{e} \mu_{t(e)}\bigl(S_{e}^{-1}(A_u)\bigr)  \ \biggr)_{u \in V}, 
\end{equation}
for all Borel sets $(A_u)_{u \in V}\subset (\mathbb{R}^n)^{\#V}$, with $(\supp (\mu_u))_{u \in V} = (F_u)_{u \in V}$, see Proposition 3, \cite{Paper_Wang}. For the $1$-vertex case see Theorem 2.8, \cite{Book_KJF1}.

The non-empty compact sets $(F_u)_{u \in V}$ of Equation (\ref{Theorem 1 b}) are often referred to as the list of attractors or self-similar sets of the IFS and the Borel probability measures, $(\mu_u)_{u \in V}$, of Equation (\ref{Theorem 3 a}), as the self-similar measures.

The \emph{open set condition} (OSC) is satisfied if and only if there exist non-empty bounded open sets $(U_u)_{u\in V} \subset  (\mathbb{R}^n)^{\#V}$, with for each $u\in V$, $S_e(U_{t(e)}) \subset U_u$ for all $e \in E_u^1$ and $S_e(U_{t(e)})\cap S_f(U_{t(f)}) = \emptyset$ for all $e,f\in E_u^1,\textrm{ with } e \neq f$. See \cite{Book_Edgar2}, \cite{Book_KJF2} or \cite{Paper_Hutchinson}.

For a set $A\subset \mathbb{R}^n$ we use the notation $C(A)$ for the convex hull of $A$, and $A^\circ$ for the interior of $A$. 

The \emph{convex strong separation condition} (CSSC) is satisfied if and only if for each $u\in V$, $S_e(C(F_{t(e)}))\cap S_f(C(F_{t(f)})) = \emptyset$ for all $e,f\in E_u^1$, with $e \neq f$. 

If the CSSC holds then the OSC is satisfied by the convex open sets $(C(F_u)^\circ)_{u \in V}$, provided $C(F_u)^\circ\neq \emptyset$ for each $u \in V$. If however $C(F_u)^\circ =\emptyset$ for some $u \in V$ then we may always reduce the dimension $n$, of the parent space $\mathbb{R}^n$, in which the IFS is constructed.

The next theorem gives the dimension of the self-similar sets provided the OSC holds, see Theorem 3, \cite{Paper_Mauldin_Williams} and for the $1$-vertex case see Theorem 9.3, \cite{Book_KJF2}. For a set $A\subset \mathbb{R}^n$, we use the usual notation $\mathcal{H}^s(A)$ for the $s$-dimensional Hausdorff measure, $\dimH A$ for the Hausdorff dimension and $\dimB A$ for the box-counting dimension. 
\begin{thm} \label{Theorem 2}
Let $\bigl(V,E^*,i,t,r,((\mathbb{R}^n,\left| \ \  \right|))_{v \in V},(S_e)_{e \in E^1}\bigr)$ be a directed graph IFS and  $(F_u)_{u \in V}$ the unique list of attractors. Let $m=\#V$ and let $\bA(t)$ denote the $m\times m$ matrix 
whose $uv$th entry is
\begin{equation*}
A_{uv}(t) = \sum_{e\in E_{uv}^1} r_e^t.
\end{equation*}
Let $\rho\left(\bA(t)\right)$ be the spectral radius of $\bA(t)$, and let $s$ be the unique non-negative real number that is the solution of $\rho\left(\bA(t)\right)=1$.  

If the OSC is satisfied then, for each $u \in V$, $s = \dimH F_u = \dimB F_u$ and $0 < \mathcal{H}^s \left(F_u\right) < +\infty$.
\end{thm} 

\section{A density result}\label{three}

In this section we consider an IFS $\bigl(V,E^*,i,t,r,p,((\mathbb{R}^n,\left| \ \  \right|))_{v \in V},(S_e)_{e \in E^1}\bigr)$, which satisfies the OSC, so the conclusions of Theorem \ref{Theorem 2} all hold for the list of attractors $(F_u)_{u \in V}$. Our aim is to prove the density result of Corollary \ref{2GcorG}. The directed graph is strongly connected so the non-negative matrix $\bA(s)$ is irreducible. By the Perron-Frobenius Theorem, see \cite{Book_Seneta}, we take $\bh=(h_v)_{v \in V}^T$ to be the positive eigenvector, which is unique up to a scaling factor, such that $\bA(s)\bh = \rho\left(\bA(s)\right) \bh = \bh$. 

We explicitly define the probability function $p:E^*\to (0,1)$, for each path $\be \in E^*$, as
$p_{\be} = h_{i(\be)}^{-1}r_{\be}^sh_{t(\be)}$. Since $\sum_{ \substack{ e\in E_u^1}  } p_e = \sum_{ \substack{ e\in E_u^1}  } h_u^{-1}r_e^sh_{t(e)} = h_u^{-1}(\bA (s) \bh)_u=h_u^{-1}h_u=1$, at each vertex $u \in V$, this defines a valid probability function for the graph, see Section \ref{two}. 

For the unique list of self-similar measures $(\mu_u)_{u\in V}$, Equation (\ref{Theorem 3 a}) is now  
\begin{equation}
\label{2Gself_similar_measure_2}
(\mu_u(A_u))_{u \in V} = \biggl( \ \sum_{ \substack{ e\in E_u^1}  } h_u^{-1}r_e^sh_{t(e)} \mu_{t(e)}(S_e^{-1}(A_u)) \ \biggr)_{u \in V},  
\end{equation}
for all Borel sets $(A_u)_{u \in V}\subset (\mathbb{R}^n)^{\#V}$.

Let $\bv$ and $\bw$ be two real $n$-dimensional (column) vectors, then $\bv \leq \bw$ if and only if $v_i \leq w_i$ for all i, $1 \leq i \leq n$, and similarly for $\bv < \bw$. 
\begin{lem}
\label{2GlemA}
Let $\bM$ be a non-negative irreducible $n\times n$ matrix with spectral radius $\rho(\bM)=1$. Suppose $\bv=(v_1, v_2,  \ldots ,  v_n)^T$ is a positive vector such that $\mathbf{0}< \bv \leq \bM \bv$, then $\bv = \bM \bv$.
\end{lem}
\begin{proof} This follows from standard Perron-Frobenius theory, see \cite{Book_Seneta}.
\end{proof}

\begin{lem}
\label{2GlemB} 
$ (\mathcal{H}^s(F_v))_{v \in V}^T$ is the unique (up to scaling) positive eigenvector of the matrix $\bA(s)$, that is 
\begin{equation*}
\bA(s) (\mathcal{H}^s(F_v))_{v \in V}^T= (\mathcal{H}^s(F_v))_{v \in V}^T.
\end{equation*}
\end{lem}
\begin{proof}$\mathcal{H}^s(F_u) \leq \sum_{ \substack{ e\in E_u^1}  } \mathcal{H}^s(S_e(F_{t(e)})) = \sum_{ \substack{ e\in E_u^1}  } r_e^s\mathcal{H}^s(F_{t(e)}) = \bigl(\bA(s) (\mathcal{H}^s(F_v))_{v \in V}^T\bigr)_u$, so $(\mathcal{H}^s(F_v))_{v \in V}^T$ is a positive vector for which $\mathbf{0}< (\mathcal{H}^s(F_v))_{v \in V}^T \leq \bA(s) (\mathcal{H}^s(F_v))_{v \in V}^T$. The matrix $\bA(s)$ is non-negative and irreducible with spectral radius $\rho(\bA(s))=1$. Applying Lemma \ref{2GlemA} completes the proof. 
\end{proof}

Given Lemma \ref{2GlemB} we put $\bh=(h_v)_{v \in V}^T=(\mathcal{H}^s(F_v))_{v \in V}^T$ to denote the eigenvector of $\bA(s)$, using any of these notations as appropriate from now on. The next lemma states that the self-similar measures of Equation (\ref{2Gself_similar_measure_2}) are in fact restricted normalised Hausdorff measures. 
\begin{lem}
\label{2GlemC}
For each $u \in V$,
\begin{equation*}
\mu_u(A)=\frac{\mathcal{H}^s(F_u\cap A )}{\mathcal{H}^s(F_u)}=h_u^{-1}\mathcal{H}^s(F_u\cap A),
\end{equation*}
for all Borel sets $A\subset\mathbb{R}^n$.
\end{lem}
\begin{proof}
This follows by a routine verification of Equation (\ref{2Gself_similar_measure_2}) for the list of restricted normalised Hausdorff measures $(h_u^{-1}\mathcal{H}^s(F_u  \cap A_u ))_{u \in V}$, where $(A_u)_{u \in V} \subset (\mathbb{R}^n)^{\#V}$ are any Borel sets. See Lemma 3.2.3 \cite{phdthesis_Boore} or \cite{Paper_Wang} for details.
\end{proof}
The notion of an $s$-straight set provides a useful intermediate step in the argument that follows, see \cite{Paper_Delaware}. A set $B\subset \mathbb{R}^n$ is \emph{$s$-straight} if
\begin {equation*}
\mathcal{H}_\infty^s(B)=\mathcal{H}^s(B) < +\infty.
\end {equation*}
Here $\mathcal{H}_\infty^s(B)=\inf\left\{\sum_{i=1}^{\infty} \left|U_i\right|^s : \{U_i\} \textrm{ is a cover of } B   \right\}$ is the Hausdorff $s$-content where there is no restriction on the diameters of the covering sets. 

\begin{lem}
\label{2GlemD}
If $B\subset \mathbb{R}^n$ is $s$-straight then $\mathcal{H}^s(A) \leq \left|A\right|^s$, for all $\mathcal{H}^s$-measurable subsets $A\subset B$.
\end{lem}
\begin{proof}
For a contradiction we assume there is an $\mathcal{H}^s$-measurable subset $A\subset B$ such that
$0 < \left|A\right|^s < \mathcal{H}^s(A) - \varepsilon$, for some $\varepsilon>0$. We may find a cover $\left\{U_i\right\}$ of $B \setminus A $ with $\sum_{i=1}^\infty\left|U_i\right|^s \leq  \mathcal{H}^s(B\setminus A) + \frac{\varepsilon}{2}$. It follows that $B \subset A \bigcup \left(\bigcup_{i=1}^\infty U_i\right)$, and so $\mathcal{H}_\infty^s(B) \leq \left|A\right|^s + \sum_{i=1}^\infty \left|U_i\right|^s \leq  \left|A\right|^s + \mathcal{H}^s(B\setminus A) + \frac{\varepsilon}{2}  
<  \mathcal{H}^s(A) - \varepsilon + \mathcal{H}^s(B\setminus A) + \frac{\varepsilon}{2} =  \mathcal{H}^s(B) - \frac{\varepsilon}{2}$. This implies $\mathcal{H}_\infty^s(B) \neq \mathcal{H}^s(B)$ which is a contradiction.
\end{proof}
We remind the reader that in this section $(F_u)_{u \in V}$ is the unique list of attractors of a directed graph IFS, $\bigl(V,E^*,i,t,r,((\mathbb{R}^n,\left| \ \  \right|))_{v \in V},(S_e)_{e \in E^1}\bigr)$, for which the OSC holds.
\begin{lem}
\label{2GlemF}
$F_u$ is $s$-straight for all $u \in V$.
\end{lem}
\begin{proof}
For a contradiction assume there exists $u \in V$ with $0 \leq \mathcal{H}_\infty^s(F_u) < \mathcal{H}^s(F_u)$.

Consider a vertex $v \in V$, $v \neq u$. As the graph is strongly connected we can always find a path $\be$ from the vertex $v$ to $u$,  and suppose such a path has length $m$, then $F_v = \bigcup_{ \substack{\be \in E_v^m}  } S_\be (F_{t(\be)})$. This implies $\mathcal{H}_\infty^s(F_v) \leq  \sum_{ \substack{\be \in E_v^m} } \mathcal{H}_\infty^s(S_\be (F_{t(\be)})) = \sum_{ \substack{\be \in E_v^m} } r_{\be}^s\mathcal{H}_\infty^s(F_{t(\be)}) < \sum_{ \substack{\be \in E_v^m} } r_{\be}^s\mathcal{H}^s(F_{t(\be)})$, where the strict inequality follows by our initial assumption, as $t(\be)=u$ for at least one path $\be \in E_v^m$. Applying Lemma \ref{2GlemB} gives 
\begin{equation*}
\mathcal{H}_\infty^s(F_v)< \sum_{ \substack{\be \in E_v^m} } r_{\be}^s\mathcal{H}^s(F_{t(\be)})=\bigl(\bA(s)^m (\mathcal{H}^s(F_w))_{w \in V}^T\bigr)_v=\mathcal{H}^s(F_v).
\end{equation*}
This argument may be repeated for any vertex, so $0 \leq \mathcal{H}_\infty^s(F_v) < \mathcal{H}^s(F_v)$, for all $v \in V$.

Let $h_{\textrm{max}}=\max\left\{h_v : v \in V\right\}$ and let $\varepsilon>0$ be given by
\begin{equation}
\label{vareps}
\varepsilon=\min \left\{ \  \frac{h_{\textrm{max}}}{2}, \ \min \left\{\frac{\mathcal{H}^s(F_v)-\mathcal{H}_\infty^s(F_v)}{2} :  v \in V\right\} \  \right\}.
\end{equation}
For each $v \in V$, we may choose some cover $\left\{U_{v,i}\right\}$ of $F_v$, with no diameter restriction, such that $\sum_{i=1}^\infty \left|U_{v,i}\right|^s < \mathcal{H}_\infty^s(F_v) + \varepsilon \leq \mathcal{H}^s(F_v) - \varepsilon$. For a given $\delta>0$, we may choose $k \in \mathbb{N}$ large enough so that $F_u = \bigcup_{ \substack{\be \in E_u^k}  } S_\be (F_{t(\be)})\subset \bigcup_{ \substack{\be \in E_u^k}  } S_\be \bigl(  \bigcup_{i=1}^\infty U_{t(\be),i}  \bigr)= \bigcup_{ \substack{\be \in E_u^k}  } \bigcup_{i=1}^\infty S_\be(U_{t(\be),i})$
where the last term is a $\delta$-cover of $F_u$. By Lemma  \ref{2GlemB}, $\sum_{ \substack{\be \in E_u^k} } h_u^{-1}r_{\be}^sh_{t(\be)}=h_u^{-1}\bigl(\bA(s)^k \bh \bigr)_u=1$, so $\sum_{ \substack{\be \in E_u^k} } r_{\be}^s\mathcal{H}^s(F_{t(\be)})=\sum_{ \substack{\be \in E_u^k} } r_{\be}^s h_{t(\be)}= h_u$ and $\sum_{ \substack{\be \in E_u^k} } r_{\be}^s \geq \frac{h_u}{h_{\textrm{max}}}$.

These results imply
\begin{equation*}
\mathcal{H}_\delta^s(F_u) \leq  \sum_{ \substack{\be \in E_u^k} } \sum_{i=1}^\infty \left|S_\be(U_{t(\be),i})\right|^s < \sum_{ \substack{\be \in E_u^k} } r_{\be}^s\bigl(\mathcal{H}^s(F_{t(\be)}) - \varepsilon \bigr) \leq h_u\Bigl(1-\frac{\varepsilon}{h_{\textrm{max}}}\Bigr).
\end{equation*}
From the choice of $\varepsilon$ in (\ref{vareps}), $0< \varepsilon \leq \frac{h_{\textrm{max}}}{2}$, which ensures $\frac{1}{2}\leq \bigl(1-\frac{\varepsilon}{h_{\textrm{max}}}\bigr)<1$ and as this argument holds for any $\delta$ we may conclude that $\mathcal{H}^s(F_u) \leq h_u\bigl(1-\frac{\varepsilon}{h_{\textrm{max}}}\bigr) < h_u = \mathcal{H}^s(F_u)$, which is the required contradiction. 
\end{proof}

\begin{cor}
\label{2GcorG}\textup{(a)} \quad $\mathcal{H}^s(A) \leq \left|A\right|^s$  for all $\mathcal{H}^s$-measurable subsets $A\subset F_u$, 
\begin{equation*}
\quad \quad \ \textup{(b)} \quad  \sup \biggl\{\frac{\mathcal{H}^s(A)}{\left|A\right|^s}  :  A \textup{ is }\mathcal{H}^s \textup{-measurable}, \, A\subset F_u\biggr\}=1. 
\end{equation*}
\end{cor}

\begin{proof}
(a) is an immediate consequence of Lemma \ref{2GlemD} and Lemma \ref{2GlemF}.

(b) Let $\alpha= \sup \bigl\{\frac{\mathcal{H}^s(A)}{\left|A\right|^s}  :  A \textup{ is }\mathcal{H}^s \textup{-measurable}, \, A\subset F_u\bigr\}$, then from part (a), $\alpha \leq 1$. It remains to show that $\alpha \geq 1$. 

Given $\varepsilon>0$ we can find a cover $\left\{U_i\right\}$ of $F_u$, such that $\sum_{i=1}^\infty\left|U_i\right|^s < \mathcal{H}_\infty^s(F_u) + \varepsilon = \mathcal{H}^s(F_u) + \varepsilon$, by Lemma \ref{2GlemF}. Each set $U_i$ is contained in a closed set of the same diameter, so we may assume that the cover consists of closed sets which are $\mathcal{H}^s$-measurable. Also $F_u\cap U_i$ is a Borel set and so is $\mathcal{H}^s$-measurable, for each $i\in \mathbb{N}$. As $F_u\subset \bigcup_{i=1}^\infty U_i$, we obtain,
\begin{equation*}
\mathcal{H}^s(F_u) \leq \sum_{i=1}^\infty \mathcal{H}^s\bigl(  F_u\cap U_i  \bigr) \leq   \sum_{i=1}^\infty \alpha \left| F_u\cap U_i \right|^s \leq \alpha \sum_{i=1}^\infty  \left| U_i \right|^s < \alpha \bigl( \mathcal{H}^s(F_u) + \varepsilon \bigr).
\end{equation*}
This argument holds for any $\varepsilon>0$, so we conclude that $\mathcal{H}^s(F_u)\leq \alpha\mathcal{H}^s(F_u)$, and this, as $0<\mathcal{H}^s(F_u)<+\infty$, implies that $\alpha\geq 1$.
\end{proof}

\section{The exact Hausdorff measure of attractors of a class of 2-vertex directed graph IFSs}\label{four}

There are few classes of sets for which the exact Hausdorff measure is known so the work of this section is of interest because, in Theorem \ref{2GthmN}, we give sufficient conditions for the calculation of the Hausdorff measure of both of the attractors of a class of $2$-vertex IFSs defined on $\mathbb{R}$ and illustrated in Figure \ref{P2Vertexb1}.
\begin{figure}[htb]
\begin{center}
\includegraphics[trim = 10mm 185mm 10mm 15mm, clip, scale =0.7]{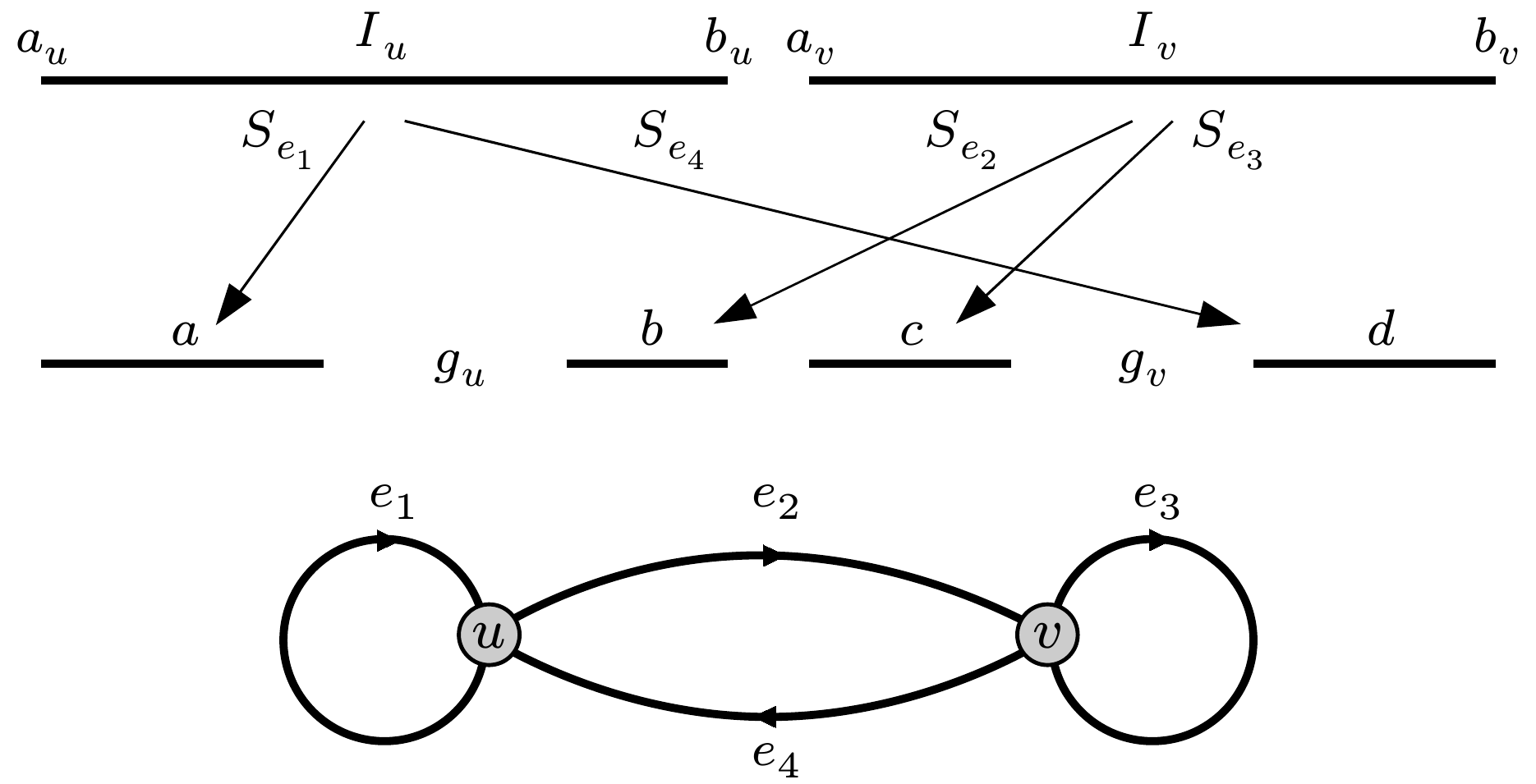}
\end{center}
\caption{A $2$-vertex directed graph IFS defined on $\mathbb{R}$, the similarities $S_{e_1},\, S_{e_2},\, S_{e_3}$ and $S_{e_4}$ do not reflect.}
\label{P2Vertexb1}
\end{figure}
We define $I_u$, $I_v$, as the smallest closed intervals containing the attractors $F_u$, $F_v$, so $C(F_u)=I_u$, $\left\{a_u, b_u\right\}\subset F_u \subset I_u=[a_u, b_u]$, with $\left|F_u\right|=\left|I_u\right|=b_u-a_u$, and similarly at the vertex $v$. We assume that all the similarities represented in diagrams in this paper preserve orientation, that is they do not involve reflections. This means that we may completely define directed graph IFSs by the use of diagrams. The strictly positive numbers, $a,\, g_u,\, b,\, c,\, g_v,\, d$, $\left|I_u\right|=a+g_u+b$, $\left|I_v\right|=c+g_v+d$, are as illustrated in Figure \ref{P2Vertexb1}, and $s=\dimH F_u=\dimH F_v$, denotes the Hausdorff dimension of the attractors. Since the gap lengths $g_u, \,g_v,$ are strictly positive the CSSC holds. The contracting similarity ratios of the similarities are given by
\begin{equation}
\label{r_{e_1}}
\begin{split}
&r_{e_1}=\frac{\left|S_{e_1}(I_u)\right|}{\left|I_u\right|}=\frac{a}{\left|I_u\right|}, \quad r_{e_2}=\frac{\left|S_{e_2}(I_v)\right|}{\left|I_v\right|}=\frac{b}{\left|I_v\right|}, \\
&r_{e_3}=\frac{\left|S_{e_3}(I_v)\right|}{\left|I_v\right|}=\frac{c}{\left|I_v\right|}, \quad r_{e_4}=\frac{\left|S_{e_4}(I_u)\right|}{\left|I_u\right|}=\frac{d}{\left|I_u\right|}.
\end{split}
\end{equation}
The similarities, $S_{e_i}: \mathbb{R} \to  \mathbb{R}$, $1 \leq i \leq 4$, are defined as
\begin{equation}
\label{S_{e_i}}
\begin{split}
&S_{e_1}(x)=r_{e_1}(x-a_u)+a_u, \quad S_{e_2}(x)=r_{e_2}(x-a_v)+a_u+a+g_u, \\
&S_{e_3}(x)=r_{e_3}(x-a_v)+a_v, \quad S_{e_4}(x)=r_{e_4}(x-a_u)+a_v+c+g_v,
\end{split}
\end{equation} 
as illustrated in Figure \ref{P2Vertexb1}.

The arguments we use in this section are based on those given by Ayer and Strichartz in \cite{Paper_Ayer_Strichartz} for $1$-vertex IFSs, particularly Lemmas 2.1, 3.1, 4.1 and Theorem 4.2 of that paper but the arguments for directed graph IFSs are much more involved. See also Theorem 7.1, \cite{Paper_Marion}. 

We reserve the letter $J$ to denote a closed interval in all that follows. The \emph{density of an interval} $J\subset I_u$, is defined as 
\begin{equation*}
d_u(J) = \frac{\mu_u(J)}{\left|J\right|^s}=\frac{\mathcal{H}^s( F_u\cap J )}{\mathcal{H}^s(F_u)\left|J\right|^s},
\end{equation*}
and for $J\subset I_v$, as
\begin{equation*}
d_v(J)=\frac{\mu_v(J)}{\left|J\right|^s}=\frac{\mathcal{H}^s( F_v\cap J )}{\mathcal{H}^s(F_v)\left|J\right|^s}.
\end{equation*}
The \emph{maximum density} for the intervals of $F_u$ is the number $\sup \bigl\{d_u(J)  :  J\subset I_u\bigr\}$, and for the intervals of $F_v$ is $\sup \bigl\{d_v(J)  :  J\subset I_v\bigr\}$.

We now prove a series of technical lemmas which lead up to Theorem \ref{2GthmN}, starting with an immediate consequence of Corollary \ref{2GcorG} of the preceding section.
 
\begin{lem}
\label{2GlemH}
For the $2$-vertex IFS of Figure \ref{P2Vertexb1}, 
\begin{equation*}
\sup \bigl\{d_u(J)  :  J\subset I_u\bigr\}=\frac{1}{\mathcal{H}^s(F_u)}\geq \frac{1}{\left|I_u\right|^s}, \quad \sup \bigl\{d_v(J)  :  J\subset I_v\bigr\}=\frac{1}{\mathcal{H}^s(F_v)}\geq \frac{1}{\left|I_v\right|^s}.
\end{equation*}
\end{lem}

In Lemma \ref{2GlemI} we collect together some useful densities for future reference. We use the eigenvector notation established in Section \ref{three}, with $\bh = ( h_u, h_v )^T = ( \mathcal{H}^s(F_u), \mathcal{H}^s(F_v) )^T$. 
\begin{lem}
\label{2GlemI}
For the $2$-vertex IFS of Figure \ref{P2Vertexb1}, 
\begin{align*}
&\textup{(a)} \ d_u(I_u)=d_u(S_{e_1}(I_u))=\frac{1}{\left|I_u\right|^s},  &&\textup{(b)} \ d_u(S_{e_2}(I_v))=\frac{h_v}{h_u}\frac{1}{\left|I_v\right|^s},\\
&\textup{(c)} \ d_v(I_v)=d_v(S_{e_3}(I_v))=\frac{1}{\left|I_v\right|^s},  &&\textup{(d)} \ d_v(S_{e_4}(I_u))=\frac{h_u}{h_v}\frac{1}{\left|I_u\right|^s},\\
&\textup{(e)} \ J\subset S_{e_1}(I_u), \, d_u(S_{e_1}^{-1}(J))=d_u(J),  &&\textup{(f)} \ J\subset S_{e_2}(I_v), \, \frac{h_v}{h_u}d_v(S_{e_2}^{-1}(J))=d_u(J),\\
&\textup{(g)} \ J\subset S_{e_3}(I_v), \, d_v(S_{e_3}^{-1}(J))=d_v(J),  &&\textup{(h)} \ J\subset S_{e_4}(I_u), \, \frac{h_u}{h_v}d_u(S_{e_4}^{-1}(J))=d_v(J).
\end{align*}
\end{lem}
\begin{proof}
We prove (h), the other parts can be proved in much the same way.
\begin{align*}
\frac{h_u}{h_v}d_u(S_{e_4}^{-1}(J))&=\frac{\mathcal{H}^s(F_u)}{\mathcal{H}^s(F_v)}\frac{\mathcal{H}^s\bigl( F_u\cap S_{e_4}^{-1}(J) \bigr)}{\mathcal{H}^s(F_u)\left|S_{e_4}^{-1}(J)\right|^s}=\frac{\mathcal{H}^s\bigl( S_{e_4}^{-1}(S_{e_4}(F_u)\cap J)\bigr) }{\mathcal{H}^s(F_v)\left|S_{e_4}^{-1}(J)\right|^s} \\
&=\frac{r_{e_4}^{-s}\mathcal{H}^s\bigl( S_{e_4}(F_u)\cap J)\bigr) }{\mathcal{H}^s(F_v)r_{e_4}^{-s}\left|J\right|^s}=\frac{\mathcal{H}^s( F_v\cap J)) }{\mathcal{H}^s(F_v)\left|J\right|^s}=d_v(J).
\qedhere 
\end{align*}
\end{proof}
The value of $\frac{h_v}{h_u}$ can be calculated using Lemma \ref{2GlemB}, which states that  
\begin{equation}
\label{matrixh_uh_v}
\left(\begin{array}{cc} 
r_{e_1}^s & r_{e_2}^s \\
r_{e_4}^s & r_{e_3}^s 
\end{array}\right)
\left(\begin{array}{c} 
h_u  \\
h_v \\
\end{array}\right)=
\left(\begin{array}{c} 
h_u  \\
h_v \\
\end{array}\right)
\end{equation} 
and this implies 
\begin{equation}
\label{h_vh_u}
\frac{h_v}{h_u}=\frac{1-r_{e_1}^s}{r_{e_2}^s}.
\end{equation}

In Lemma \ref{2GlemK} there is a good reason for the choice of functions $f_u$ and $f_v$. If we were instead to use $l_u(x,y)=\frac{x^s + y^s}{(x + g_u + y)^s}$ and $l_v(x,y)=\frac{x^s + y^s}{(x + g_v + y)^s}$ then, in order to obtain $l_u(a,b),l_v(c,d) \leq 1$, we would require $\frac{h_u\left|I_v\right|^s}{h_v\left|I_u\right|^s}=1$ and this is a much more restrictive condition than $(1)$. Also it is not obvious how such a condition could be checked.

\begin{lem}
\label{2GlemK} For the $2$-vertex IFS of Figure \ref{P2Vertexb1}, let 
\begin{align*}
P&=\left\{(x,y)  :  0\leq x \leq a, \, 0\leq y \leq b\right\} \setminus \left\{(a,b)\right\}, \\
Q&=\left\{(x,y)  :  0\leq x \leq c, \, 0\leq y \leq d\right\} \setminus \left\{(c,d)\right\}, \\
f_u(x,y)&=\frac{x^s + \frac{h_v}{h_u}y^s}{(x + g_u + y)^s},  \quad \textrm{and} \quad f_v(x,y)=\frac{x^s + \frac{h_u}{h_v}y^s}{(x + g_v + y)^s}.
\end{align*}

Suppose the following three conditions hold,
\begin{equation*}
\textup{(1)}  \quad \left|I_u\right|=\left|I_v\right|, \quad \textup{(2)}  \quad \frac{h_v}{h_u} \leq 1, \quad \textup{(3)}  \quad \frac{(a+g_u)(\left|I_u\right|^s-a^s)}{ba^s}\geq 1,
\end{equation*}
then
\begin{align*}
&\textup{(a)} \quad f_u(a,b)=f_v(c,d)=1, \\
&\textup{(b)} \quad f_u(x,y)<1, \textrm{for all } (x,y) \in P, \\
&\textup{(c)} \quad f_v(x,y)<1, \textrm{for all } (x,y) \in Q.
\end{align*}
\end{lem}
\begin{proof}(a) From the definition of $f_u$,
\begin{align*}
f_u(a,b) &=\frac{a^s + \frac{h_v}{h_u}b^s}{(a + g_u + b)^s} = \frac{\left|I_u\right|^s r_{e_1}^s + \frac{h_v}{h_u}\left|I_v\right|^s r_{e_2}^s}{\left|I_u\right|^s} && (\textrm{by (\ref{r_{e_1}})})\\ 
&=\frac{1}{h_u} \biggl(r_{e_1}^sh_u + \frac{\left|I_v\right|^s}{\left|I_u\right|^s}r_{e_2}^sh_v\biggr) = \frac{1}{h_u}\bigl(r_{e_1}^sh_u + r_{e_2}^sh_v\bigr) && (\textrm{by (1)}) \\
&=1 && (\textrm{by (\ref{matrixh_uh_v})}).  
\end{align*} 
In the same way it can be shown that $f_v(c,d)=1$.

Parts (b) and (c) can be verified using calculus. To give a rough idea of the type of argument involved, let $y_{\textrm{max}}$ be the point at which the maximum value of $f_u(a,y)$ occurs. It can be shown that $\frac{y_{\textrm{max}}}{b}=\bigl(\frac{(a+g_u)(\left|I_u\right|^s-a^s)}{ba^s}\bigr)^\frac{1}{1-s}$, so if $(3)$ holds $\frac{y_{\textrm{max}}}{b} \geq 1$. As $f_u(a,y)$ strictly increases up to $y_{\textrm{max}}$ and $f_u(a,b)=1$, it follows that $f_u(a,y)<1$ for all $(a,y) \in P$. See Lemma 3.4.4 \cite{phdthesis_Boore}. 
\end{proof}

\begin{figure}[htb]
\begin{center}
\includegraphics[trim = 10mm 200mm 10mm 15mm, clip, scale =0.7]{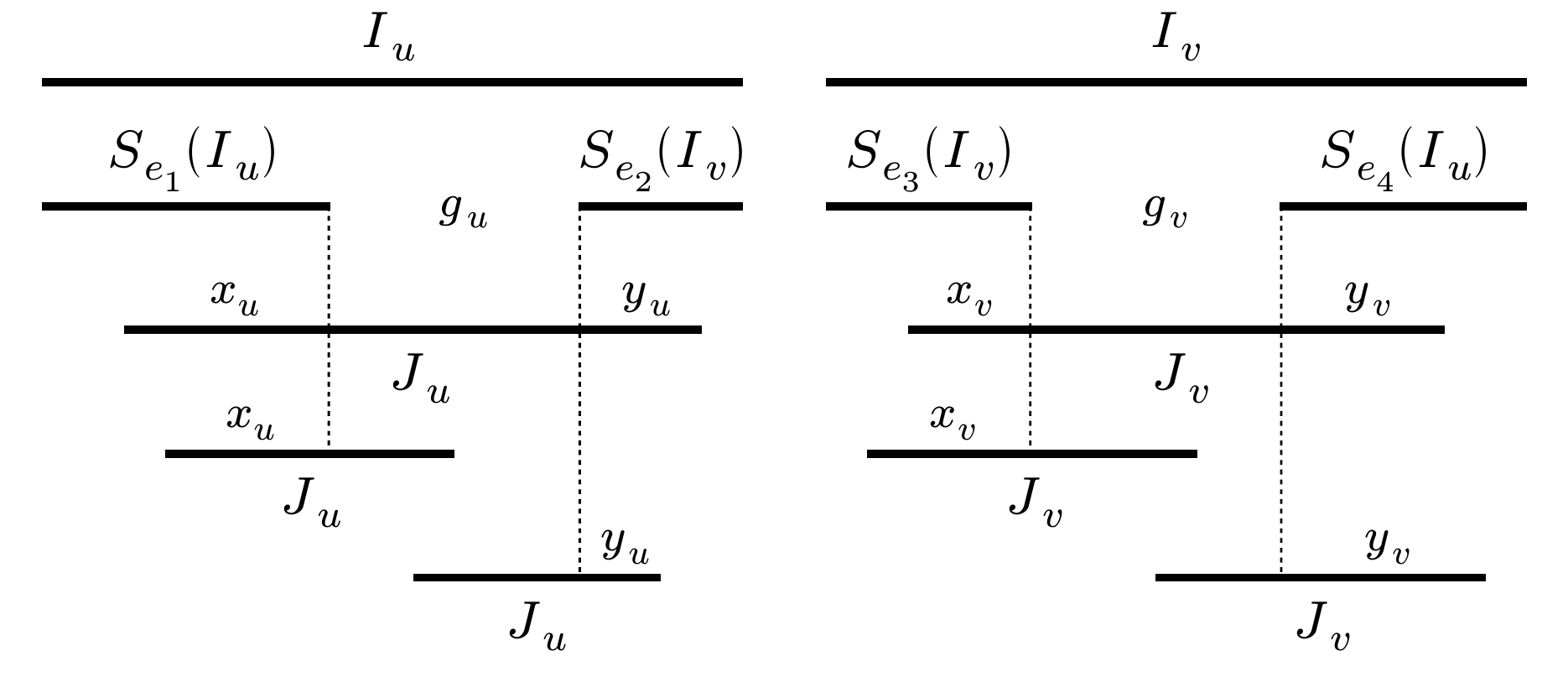}
\end{center}
\caption{The intervals $J_u$ and $J_v$.}
\label{2GJuJv}
\end{figure}

The next two lemmas give important results which we will apply in the proof of Theorem \ref{2GthmN} which follows immediately after.

\begin{lem}
\label{2GlemL} 
For the $2$-vertex IFS of Figure \ref{P2Vertexb1}, let $J_u\subset I_u$ be an interval which is not contained in the level-$1$ intervals $S_{e_1}(I_u)$, $S_{e_2}(I_v)$, with $d_u(J_u)>0$ and let $J_v\subset I_v$, be an interval which is not contained in the level-$1$ intervals $S_{e_3}(I_v)$, $S_{e_4}(I_u)$, with $d_v(J_v)>0$, as illustrated in Figure \ref{2GJuJv}. Suppose also that the conditions of Lemma \ref{2GlemK} hold.

\textup{(a)} If $J_u \neq I_u$ then $d_u(J_u)<\max\left\{\, d_u(S_{e_1}^{-1}(J_u)), \ d_v(S_{e_2}^{-1}(J_u))\, \right\}$. 

\textup{(b)} If $J_v \neq I_v$ then $d_v(J_v)<\max\left\{\, d_v(S_{e_3}^{-1}(J_v)), \ d_u(S_{e_4}^{-1}(J_v))\, \right\}$.
\end{lem}

\begin{proof} The lengths $x_u,\,y_u,\,x_v,\,y_v$, illustrated in Figure \ref{2GJuJv}, are defined as
\begin{equation*}
x_u=\left|S_{e_1}(I_u)\cap J_u\right|, \ y_u=\left|S_{e_2}(I_v)\cap J_u\right|, \ x_v=\left|S_{e_3}(I_v)\cap J_v\right|, \ y_v=\left|S_{e_4}(I_u)\cap J_v\right|,
\end{equation*}
and for convenience we put $0=\left| \, \emptyset \, \right|$, and also take the densities of the empty interval to be zero, that is $d_u(\emptyset)=d_v( \emptyset)=0$. As we are assuming $d_u(J_u)>0$, at least one of $x_u$ or $y_u$ will be strictly positive, and similarly for $x_v$ and $y_v$. 
\begin{align*}
\textrm{(a)} \quad d_u(J_u)&= \frac{\mathcal{H}^s( F_u\cap J_u )}{\mathcal{H}^s(F_u)\left|J_u\right|^s} \\
&= \frac{\mathcal{H}^s( F_u\cap (S_{e_1}(I_u)\cap J_u)) + \mathcal{H}^s( F_u\cap (S_{e_2}(I_v)\cap J_u))}{\mathcal{H}^s(F_u)\left|J_u\right|^s} \\
&= \frac{\left|S_{e_1}(I_u)\cap J_u\right|^sd_u(S_{e_1}(I_u)\cap J_u) + \left|S_{e_2}(I_v)\cap J_u\right|^sd_u(S_{e_2}(I_v)\cap J_u)}{\left|J_u\right|^s} \\
&= \frac{x_u^sd_u(S_{e_1}(I_u)\cap J_u) + y_u^sd_u(S_{e_2}(I_v)\cap J_u)}{(x_u+g_u+y_u)^s}.  
\end{align*}
Applying Lemma \ref{2GlemI}(e), (f), and Lemma \ref{2GlemK}(b), we obtain, 
\begin{align*}
d_u(J_u)&= \frac{x_u^sd_u(I_u\cap S_{e_1}^{-1}(J_u)) + \frac{h_v}{h_u}y_u^sd_v(I_v\cap S_{e_2}^{-1}(J_u))}{(x_u+g_u+y_u)^s} \\
&= \frac{x_u^sd_u(S_{e_1}^{-1}(J_u)) + \frac{h_v}{h_u}y_u^sd_v(S_{e_2}^{-1}(J_u))}{(x_u+g_u+y_u)^s} \\
&\leq \biggl( \frac{x_u^s + \frac{h_v}{h_u}y_u^s}{(x_u+g_u+y_u)^s} \biggr)\max\left\{\, d_u(S_{e_1}^{-1}(J_u)), \ d_v(S_{e_2}^{-1}(J_u))\, \right\} \\
&= f_u(x_u,y_u)\max\left\{\, d_u(S_{e_1}^{-1}(J_u)), \ d_v(S_{e_2}^{-1}(J_u))\, \right\} \\
&< \max\left\{\, d_u(S_{e_1}^{-1}(J_u)), \ d_v(S_{e_2}^{-1}(J_u))\, \right\}. 
\end{align*}
The proof of part (b) is similar to that given in part (a), applying instead Lemma \ref{2GlemI}(g), (h), and Lemma \ref{2GlemK}(c). See Lemma 3.4.5 \cite{phdthesis_Boore}.
\end{proof}

We now consider $\sup \bigl\{d_u(J)  :  S_{e_1}(I_u)\subset J \subset I_u\bigr\}$. As shown in Figure \ref{P2Vertexb1}, $I_u=[a_u, b_u]$, and $S_{e_1}(I_u)=[a_u, a_u+a]$, so
\begin{equation*}
\sup \bigl\{d_u(J)  :  S_{e_1}(I_u)\subset J\subset I_u\bigr\}=\sup\left\{\frac{\mathcal{H}^s(F_u\cap[a_u,x])}{\mathcal{H}^s(F_u)(x-a_u)^s}  :  x \in [a_u+a, b_u]\right\}.
\end{equation*}
The function $\frac{\mathcal{H}^s( F_u\cap [a_u,  x])}{\mathcal{H}^s(F_u)(x-a_u)^s}$ is a continuous function of $x$ on the compact interval $[a_u+a, b_u]$, where $a>0$, so it is bounded and attains its bound for at least one $x_0 \in [a_u+a,  b_u]$. For the largest such $x_0$, we may define an interval $L_u=[a_u,  x_0]$,  $S_{e_1}(I_u)\subset L_u \subset I_u$, which satisfies, 
\begin{equation}
\label{L_u}
d_u(L_u)=\sup \bigl\{d_u(J) : S_{e_1}(I_u)\subset J \subset I_u\bigr\}.
\end{equation}
Similarly intervals $L_v, \, R_u, \, R_v$, exist for which the following equations hold,
\begin{align}
\label{L_v}
&d_v(L_v)=\sup \bigl\{d_v(J)  :  S_{e_3}(I_v)\subset J \subset I_v\bigr\}, \\
\label{R_u}
&d_u(R_u)=\sup \bigl\{d_u(J)  :  S_{e_2}(I_v)\subset J \subset I_u\bigr\}, \\
\label{R_v}
&d_v(R_v)=\sup \bigl\{d_u(J)  :  S_{e_4}(I_u)\subset J \subset I_v\bigr\}. 
\end{align}
Some possible candidates for $L_u, \, L_v, \, R_u, \, R_v$ are illustrated in Figure \ref{2GLuRuLvRv}. 
\begin{figure}[htb]
\begin{center}
\includegraphics[trim = 10mm 225mm 10mm 15mm, clip, scale =0.7]{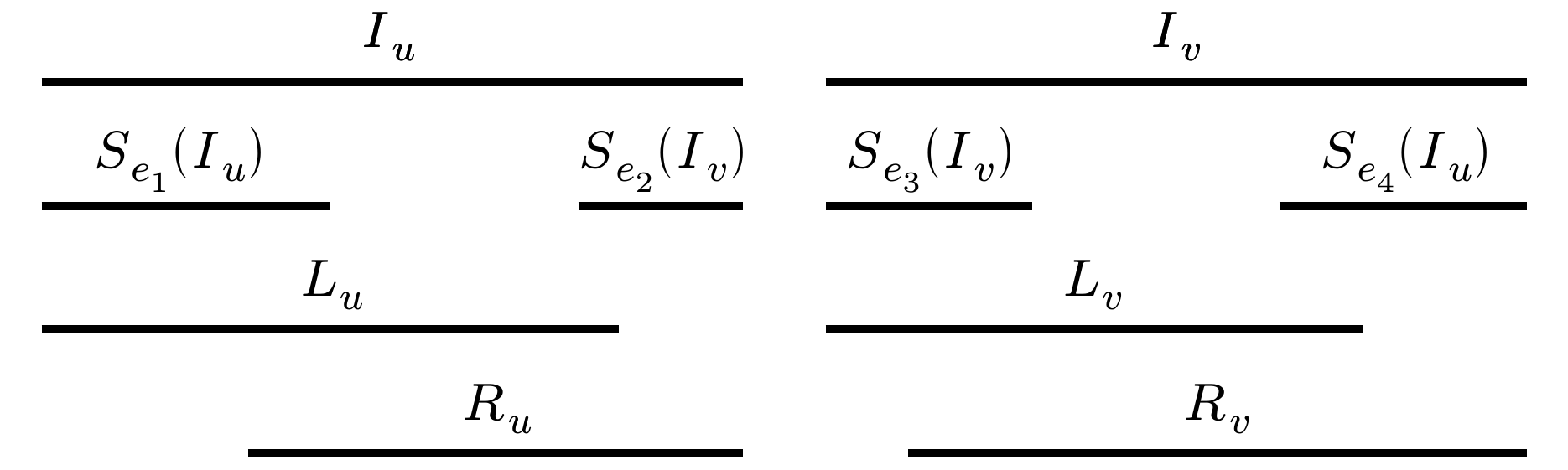}
\end{center}
\caption{Some possibilities for the intervals $L_u, \, L_v, \, R_u$, and $R_v$.}
\label{2GLuRuLvRv}
\end{figure}

\begin{lem}
\label{2GlemM}
For the $2$-vertex IFS of Figure \ref{P2Vertexb1}, let the intervals $L_u, \, L_v, \, R_u$, and $R_v$ be as defined in Equations (\ref{L_u}), (\ref{L_v}), (\ref{R_u}), and (\ref{R_v}), and suppose the conditions of Lemma \ref{2GlemK} hold. 

Then
\begin{align*}
&\textup{(a)} \ d_u(L_u)=\frac{1}{\left|I_u\right|^s},   \quad \quad \quad \quad \textup{(b)} \ d_v(L_v)=\frac{1}{\left|I_u\right|^s}, \\
&\textup{(c)} \ d_u(R_u)=\frac{1}{\left|I_u\right|^s},   \quad \quad \quad \quad \textup{(d)} \ d_v(R_v)=\frac{h_u}{h_v}\frac{1}{\left|I_u\right|^s}. 
\end{align*}
\end{lem}

\begin{proof}
(a) As stated in Lemma \ref{2GlemI}(a), $d_u(I_u)=d_u(S_{e_1}(I_u))=\frac{1}{\left|I_u\right|^s}$, which implies, from the definition of $L_u$ in Equation (\ref{L_u}), that $d_u(L_u)\geq \frac{1}{\left|I_u\right|^s}$. For a contradiction we assume $d_u(L_u)> \frac{1}{\left|I_u\right|^s}$. Clearly $S_{e_1}(I_u)\subsetneqq L_u \subsetneqq I_u$. Also if the right hand endpoint of the interval $L_u$ were to lie in the gap between the intervals $S_{e_1}(I_u)$ and $S_{e_2}(I_v)$ then $d_u(S_{e_1}(I_u))>d_u(L_u)$ which contradicts our assumption, so the right hand endpoint of $L_u$ lies in $S_{e_2}(I_v)$. This is the situation illustrated in Figure \ref{2GLu}.
\begin{figure}[htb]
\begin{center}
\includegraphics[trim = 10mm 150mm 10mm 15mm, clip, scale =0.7]{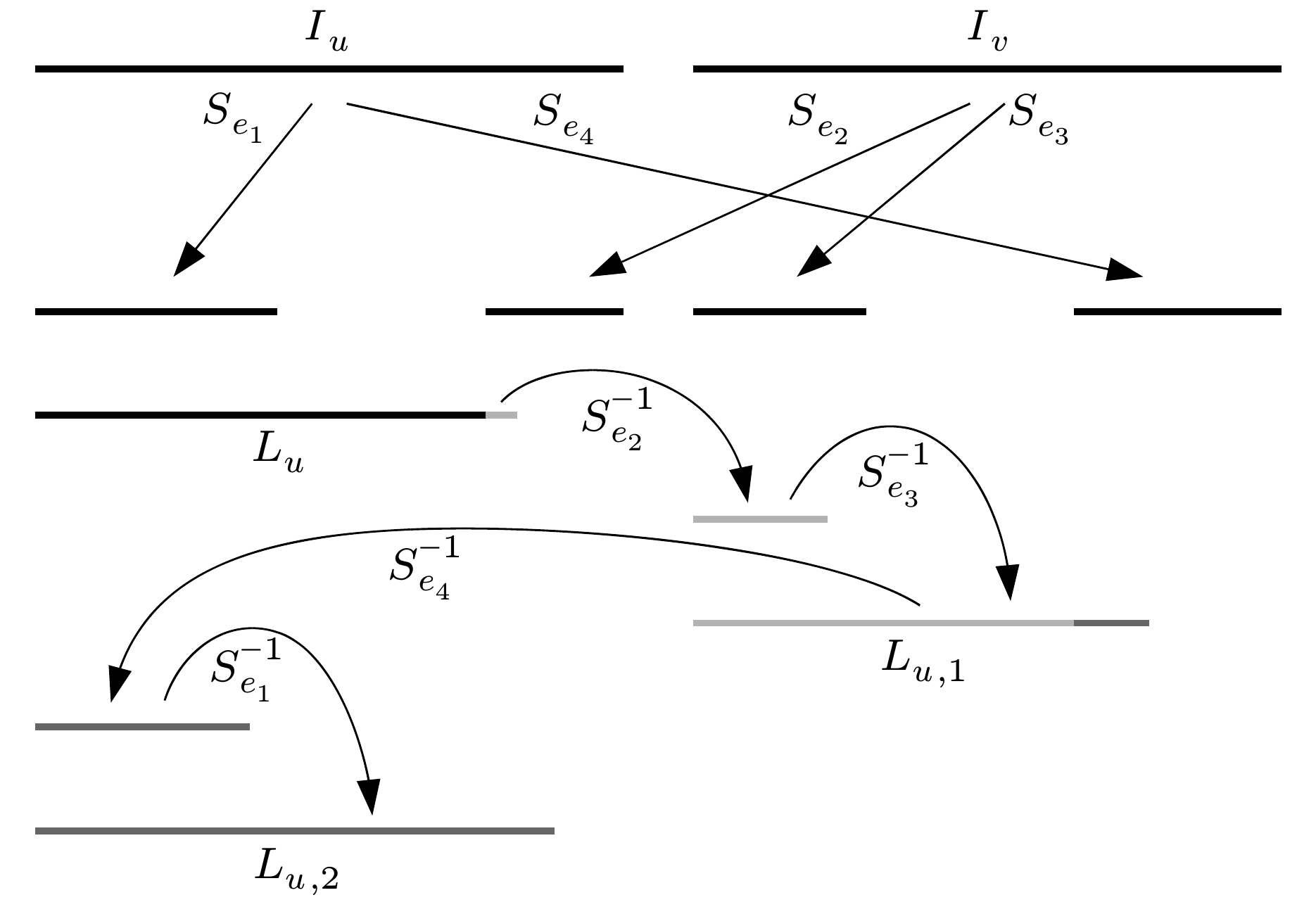}
\end{center}
\caption{The intervals $L_u, \, L_{u,1}$, and $L_{u,2}$.}
\label{2GLu}
\end{figure}
 
Applying Lemma \ref{2GlemL}(a), we obtain 
\begin{equation*}
d_u(L_u)< d_v(S_{e_2}^{-1}(L_u)),
\end{equation*}
since    $d_u(S_{e_1}^{-1}(L_u))=d_u(I_u)=\frac{1}{\left|I_u\right|^s}<d_u(L_u)$. If necessary, by repeatedly applying the expanding similarity $S_{e_3}^{-1}$ to the interval $S_{e_2}^{-1}(L_u)\cap I_v$, we must eventually arrive at an interval $L_{u,1}$, which is not contained in the interval $S_{e_3}(I_v)$, where $L_{u,1}=S_{e_3}^{-m}(S_{e_2}^{-1}(L_u)\cap I_v)$, for some $m\geq 0$. By Lemma \ref{2GlemI}(g) $d_v(S_{e_3}^{-m}(S_{e_2}^{-1}(L_u)\cap I_v))=d_v(S_{e_2}^{-1}(L_u))$ so
\begin{equation}
\label{ineq1}
d_u(L_u)< d_v(L_{u,1}).
\end{equation}
and $d_v(L_{u,1})>\frac{1}{\left|I_u\right|^s}$. Again the right hand endpoint of $L_{u,1}$ cannot lie in the gap between the intervals $S_{e_3}(I_v)$ and $S_{e_4}(I_u)$ for then, $d_v(S_{e_3}(I_v))>d_v(L_{u,1})>\frac{1}{\left|I_u\right|^s}$. This is impossible because $d_v(S_{e_3}(I_v))=\frac{1}{\left|I_u\right|^s}$, by Lemma \ref{2GlemI}(c) and condition (1) of Lemma \ref{2GlemK}. Similarly since $d_v(I_v)=\frac{1}{\left|I_u\right|^s}$, again by Lemma \ref{2GlemI}(c) and condition (1) of Lemma \ref{2GlemK}, we cannot have $L_{u,1}=I_v$. Therefore $S_{e_3}(I_v) \subsetneqq L_{u,1} \subsetneqq I_v$. The situation is shown in Figure \ref{2GLu} for $m=1$. 

Now we may apply Lemma \ref{2GlemL}(b), to obtain
\begin{equation*}
d_v(L_{u,1})<  d_u(S_{e_4}^{-1}(L_{u,1})),
\end{equation*}
as  $d_v(S_{e_3}^{-1}(L_{u,1}))=d_v(I_v)=\frac{1}{\left|I_u\right|^s}<d_v(L_{u,1})$. If necessary, by repeatedly applying the expanding similarity $S_{e_1}^{-1}$ to the interval $S_{e_4}^{-1}(L_{u,1}) \cap I_u$, we must eventually arrive at an interval $L_{u,2}$, with $S_{e_1}(I_u)\subset L_{u,2}$, where $L_{u,2}=S_{e_1}^{-n}(S_{e_4}^{-1}(L_{u,1}) \cap I_u)$, for some $n\geq 0$. The situation is illustrated in Figure \ref{2GLu} for $n=1$. By Lemma \ref{2GlemI}(e) $d_u(L_{u,2})=d_u(S_{e_4}^{-1}(L_{u,1}))$ so 
\begin{equation}
\label{ineq2}
d_v(L_{u,1})<d_u(L_{u,2}).
\end{equation}
From the definition of the interval $L_u$ in Equation (\ref{L_u}), $d_u(L_{u,2})\leq d_u(L_u)$, which together with Equations (\ref{ineq1}) and (\ref{ineq2}) gives
\begin{equation*}
d_u(L_u)<d_v(L_{u,1})<d_u(L_{u,2})\leq d_u(L_u).
\end{equation*} 
This contradiction completes the proof of part (a).

The proof of part (b) is symmetrically identical to that of part (a). The proofs of parts (c) and (d) are slightly more involved but very similar in method. See Lemma 3.4.6 \cite{phdthesis_Boore}. 
\end{proof}
The next theorem enables the calculation of the Hausdorff measure of both of the attractors of a class of $2$-vertex IFSs. 
\begin{thm}
\label{2GthmN}
For the $2$-vertex IFS of Figure \ref{P2Vertexb1}, where $s=\dimH F_u$ $=\dimH F_v$, suppose that the following conditions hold,
\begin{equation*}
\textup{(1)}  \quad \left|I_u\right|=\left|I_v\right|, \quad \textup{(2)}  \quad \frac{h_v}{h_u} \leq 1, \quad \textup{(3)}  \quad \frac{(a+g_u)(\left|I_u\right|^s-a^s)}{ba^s}\geq 1.
\end{equation*}
Then
\begin{equation*}
\mathcal{H}^s(F_u)=\left|I_u\right|^s \quad \textrm{and} \quad \mathcal{H}^s(F_v)=\left|I_u\right|^s\biggl(\frac{1-r_{e_1}^s}{r_{e_2}^s}\biggr).
\end{equation*}
\end{thm}
\begin{proof}
For any interval $J\subset I_u$, with $d_u(J)>0$, we aim to show that $d_u(J)\leq \frac{1}{\left|I_u\right|^s}$, then, by Lemma \ref{2GlemH}, the maximum density will satisfy 
\begin{equation*}
\sup \bigl\{d_u(J)  :  J\subset I_u\bigr\}=\frac{1}{\mathcal{H}^s(F_u)} = \frac{1}{\left|I_u\right|^s}.
\end{equation*}
By Lemma \ref{2GlemI}(e), for any interval $J\subset I_u$, $d_u(J)=d_u(S_{e_1}^{-1}(S_{e_1}(J)))=d_u(S_{e_1}(J))$, so it is enough to prove $d_u(J)\leq \frac{1}{\left|I_u\right|^s}$ for any interval $J$ contained in a level-$1$ interval. 

Let $J\subset I_u$ be any interval contained in one of the level-$1$ intervals $S_{e_1}(I_u)$ or $S_{e_2}(I_v)$ with $d_u(J)>0$. Operating on $J$ with the expanding similarities $S_{e_1}^{-1}$,  $S_{e_2}^{-1}$, $S_{e_3}^{-1}$,  $S_{e_4}^{-1}$, as necessary, we must eventually arrive at an interval $J_u\subset I_u$ or $J_v\subset I_v$, which is not contained in any level-$1$ interval. The situation is illustrated in Figure \ref{2GJuJv}. For $J_u\subset I_u$ the maps $S_{e_2}^{-1}$ and $S_{e_4}^{-1}$ must be applied an equal number of times to the interval $J$, and so the scaling factors of $\frac{h_v}{h_u}$ and $\frac{h_u}{h_v}$ in Lemma \ref{2GlemI}(f) and (h) will cancel each other out. This means, by Lemma \ref{2GlemI}(e), (f), (g) and (h), that $d_u(J_u)=d_u(J)$. If $J_u=I_u$, then $d_u(J_u)=\frac{1}{\left|I_u\right|^s}$, by Lemma \ref{2GlemI}(a), so we may assume $J_u\neq I_u$. Applying Lemma \ref{2GlemL}(a) gives
\begin{equation}
\label{ineqJ_u}
d_u(J)=d_u(J_u)<\max\left\{\, d_u(S_{e_1}^{-1}(J_u)), \ d_v(S_{e_2}^{-1}(J_u))\, \right\}.
\end{equation}
For $J_v\subset I_v$, the map $S_{e_2}^{-1}$ must have been applied exactly one more time to the interval $J$ than the map $S_{e_4}^{-1}$, so a factor of $\frac{h_v}{h_u}$ will occur by Lemma \ref{2GlemI}(f). This means, by Lemma \ref{2GlemI}(e), (f), (g) and (h), that $\frac{h_v}{h_u}d_v(J_v)=d_u(J)$. If $J_v=I_v$, then $\frac{h_v}{h_u}d_v(J_v)=\frac{h_v}{h_u}\frac{1}{\left|I_u\right|^s}\leq \frac{1}{\left|I_u\right|^s}$, by Lemma \ref{2GlemI}(c) and condition (2), so we may assume $J_v\neq I_v$. Applying Lemma \ref{2GlemL}(b) gives
\begin{equation}
\label{ineqJ_v}
d_u(J)=\frac{h_v}{h_u}d_v(J_v)<\frac{h_v}{h_u}\max\left\{\, d_v(S_{e_3}^{-1}(J_v)), \ d_u(S_{e_4}^{-1}(J_v))\, \right\}.
\end{equation}
We now determine upper bounds for the densities (a) $d_u(S_{e_1}^{-1}(J_u))$, (b) $d_v(S_{e_2}^{-1}(J_u))$, (c) $d_v(S_{e_3}^{-1}(J_v))$, and (d) $d_u(S_{e_4}^{-1}(J_v))$, considering each in turn.

(a) $d_u(S_{e_1}^{-1}(J_u))$.

Expanding the interval $S_{e_1}^{-1}(J_u)\cap I_u$, if necessary, we obtain an interval $J_{u,1}$, not contained in any level-$1$ interval, where one of the following two possibilites hold,
\begin{align*}
&\textrm{(i) } \quad  S_{e_2}(I_v) \subset J_{u,1}=\bigl( (S_{e_4}^{-1}\circ S_{e_2}^{-1})^m \bigr) \bigl( S_{e_1}^{-1}(J_u)\cap I_u \bigr) \subset I_u, \\
&\textrm{(ii)} \quad  S_{e_4}(I_u) \subset J_{u,1}=\bigl( S_{e_2}^{-1}\circ (S_{e_4}^{-1}\circ S_{e_2}^{-1})^n \bigl) \bigl( S_{e_1}^{-1}(J_u)\cap I_u \bigr) \subset I_v, 
\end{align*}
for $m,n\geq 0$. For (i), using Lemma \ref{2GlemI}(f) and (h), and Lemma \ref{2GlemM}(c), we obtain
\begin{equation*}
d_u(S_{e_1}^{-1}(J_u))=d_u(S_{e_1}^{-1}(J_u)\cap I_u)=d_u(J_{u,1})\leq d_u(R_u)=\frac{1}{\left|I_u\right|^s}.
\end{equation*}
For (ii), using Lemma \ref{2GlemI}(f) and (h), and Lemma \ref{2GlemM}(d), we obtain
\begin{equation*}
d_u(S_{e_1}^{-1}(J_u))=d_u(S_{e_1}^{-1}(J_u)\cap I_u)=\frac{h_v}{h_u}d_v(J_{u,1})\leq \frac{h_v}{h_u}d_v(R_v)=\frac{h_v}{h_u}\frac{h_u}{h_v}\frac{1}{\left|I_u\right|^s}=\frac{1}{\left|I_u\right|^s},
\end{equation*}
In both cases 
\begin{equation*}
d_u(S_{e_1}^{-1}(J_u))\leq \frac{1}{\left|I_u\right|^s}.
\end{equation*}

(b) $d_v(S_{e_2}^{-1}(J_u))$.

Expanding the interval $S_{e_2}^{-1}(J_u)\cap I_v$, if necessary, we obtain an interval $J_{u,1}\subset I_v$, not contained in any level-$1$ interval, with
\begin{equation*}
S_{e_3}(I_v) \subset J_{u,1}=\bigl( (S_{e_3}^{-1})^m \bigl) \bigl( S_{e_2}^{-1}(J_u)\cap I_v \bigr) \subset I_v,
\end{equation*}
for $m\geq0$. By Lemma \ref{2GlemI}(g) and Lemma \ref{2GlemM}(b),
\begin{equation*}
d_v(S_{e_2}^{-1}(J_u))=d_v(S_{e_2}^{-1}(J_u)\cap I_v)=d_v(J_{u,1})\leq d_v(L_v)=\frac{1}{\left|I_u\right|^s}.
\end{equation*}

(c) $d_v(S_{e_3}^{-1}(J_v))$.

Expanding the interval $ S_{e_3}^{-1}(J_v)\cap I_v$, if necessary, we obtain an interval $J_{v,1}$, not contained in any level-$1$ interval, where one of the following two possibilites hold,
\begin{align*}
&\textrm{(i) } \quad  S_{e_4}(I_u) \subset J_{v,1}=\bigl( (S_{e_2}^{-1}\circ S_{e_4}^{-1})^m \bigr) \bigl( S_{e_3}^{-1}(J_v)\cap I_v \bigr) \subset I_v, \\
&\textrm{(ii)} \quad  S_{e_2}(I_v) \subset J_{v,1}=\bigl( S_{e_4}^{-1}\circ (S_{e_2}^{-1}\circ S_{e_4}^{-1})^m \bigr) \bigl( S_{e_3}^{-1}(J_v)\cap I_v \bigr) \subset I_u, 
\end{align*}
for $m,n\geq 0$. For (i), using Lemma \ref{2GlemI}(f) and (h), and Lemma \ref{2GlemM}(b), gives
\begin{equation*}
d_v(S_{e_3}^{-1}(J_v))=d_v(S_{e_3}^{-1}(J_v)\cap I_v)=d_v(J_{v,1})\leq d_v(R_v)= \frac{h_u}{h_v}\frac{1}{\left|I_u\right|^s}.
\end{equation*}
For (ii), using Lemma \ref{2GlemI}(f) and (h), and Lemma \ref{2GlemM}(c), gives
\begin{equation*}
d_v(S_{e_3}^{-1}(J_v))=d_v(S_{e_3}^{-1}(J_v)\cap I_v)=\frac{h_u}{h_v}d_u(J_{v,1})\leq \frac{h_u}{h_v}d_u(R_u)= \frac{h_u}{h_v}\frac{1}{\left|I_u\right|^s}.
\end{equation*}
In both cases 
\begin{equation*}
d_v(S_{e_3}^{-1}(J_v))\leq \frac{h_u}{h_v}\frac{1}{\left|I_u\right|^s}.
\end{equation*}

(d) $d_u(S_{e_4}^{-1}(J_v))$.

Expanding the interval $S_{e_4}^{-1}(J_v)\cap I_u$, if necessary, we obtain an interval $J_{v,1}\subset I_u$, not contained in any level-$1$ interval, with
\begin{equation*}
S_{e_1}(I_u) \subset J_{v,1}=\bigl( (S_{e_1}^{-1})^m \bigr) \bigl( S_{e_4}^{-1}(J_v)\cap I_u \bigr) \subset I_u,
\end{equation*}
for $m\geq0$. By Lemma \ref{2GlemI}(e) and Lemma \ref{2GlemM}(a),
\begin{equation*}
d_u(S_{e_4}^{-1}(J_v))=d_u(S_{e_4}^{-1}(J_v)\cap I_u)=d_u(J_{v,1})\leq d_u(L_u)=\frac{1}{\left|I_u\right|^s}.
\end{equation*}

Putting the results of parts (a) and (b) into Equation (\ref{ineqJ_u}) we obtain
\begin{equation*}
d_u(J)=d_u(J_u)<\max\left\{\, d_u(S_{e_1}^{-1}(J_u)), \ d_v(S_{e_2}^{-1}(J_u))\, \right\} \leq \frac{1}{\left|I_u\right|^s}.
\end{equation*}
Putting the results of parts (c) and (d) into Equation (\ref{ineqJ_v}), remembering that by condition (2), $\frac{h_u}{h_v}\geq 1$, gives
\begin{equation*}
d_u(J)=\frac{h_v}{h_u}d_v(J_v)<\frac{h_v}{h_u}\max\left\{\, d_v(S_{e_3}^{-1}(J_v)), \ d_u(S_{e_4}^{-1}(J_v))\, \right\} \leq \frac{h_v}{h_u}\frac{h_u}{h_v}\frac{1}{\left|I_u\right|^s}=\frac{1}{\left|I_u\right|^s}.
\end{equation*}

Therefore in all cases
\begin{equation*}
d_u(J)\leq \frac{1}{\left|I_u\right|^s},
\end{equation*} 
which completes the proof that $\mathcal{H}^s(F_u)=\left|I_u\right|^s$. The expression for $\mathcal{H}^s(F_v)$ now follows immediately by Equation (\ref{h_vh_u}). 
\end{proof}
We now define the \emph{$2$-vertex IFS (on the unit interval) of Figure \ref{P2Vertexb1}} to be the $2$-vertex directed graph IFS of Figure \ref{P2Vertexb1} but with $a_u$, $b_u$, $a_v$ and $b_v$ taking the specific values $a_u=a_v=0$ and $b_u=b_v=1$,  so that $I_u=I_v=[0,1]$ and $\left|I_u\right|=\left|I_v\right|=1$. For the rest of this paper we only consider this family of $2$-vertex directed graph IFSs for which condition (1) of Theorem \ref{2GthmN} always holds. By Equations (\ref{r_{e_1}}), the contracting similarity ratios of these IFSs are 
\begin{equation}
\label{r_{e_i} unit interval}
r_{e_1}=a, \quad r_{e_2}=b, \quad r_{e_3}=c, \quad r_{e_4}=d,
\end{equation}
and by Equations (\ref{S_{e_i}}), the similarities are 
\begin{equation}
\label{S_{e_i} unit interval}
\begin{split}
&S_{e_1}(x)=r_{e_1}x, \quad S_{e_2}(x)=r_{e_2}x+a+g_u, \\
&S_{e_3}(x)=r_{e_3}x, \quad S_{e_4}(x)=r_{e_4}x+c+g_v.
\end{split}
\end{equation}
\begin{figure}[htb]
\begin{center}
\includegraphics[trim = 25mm 177mm 25mm 40mm, clip, scale =0.55]{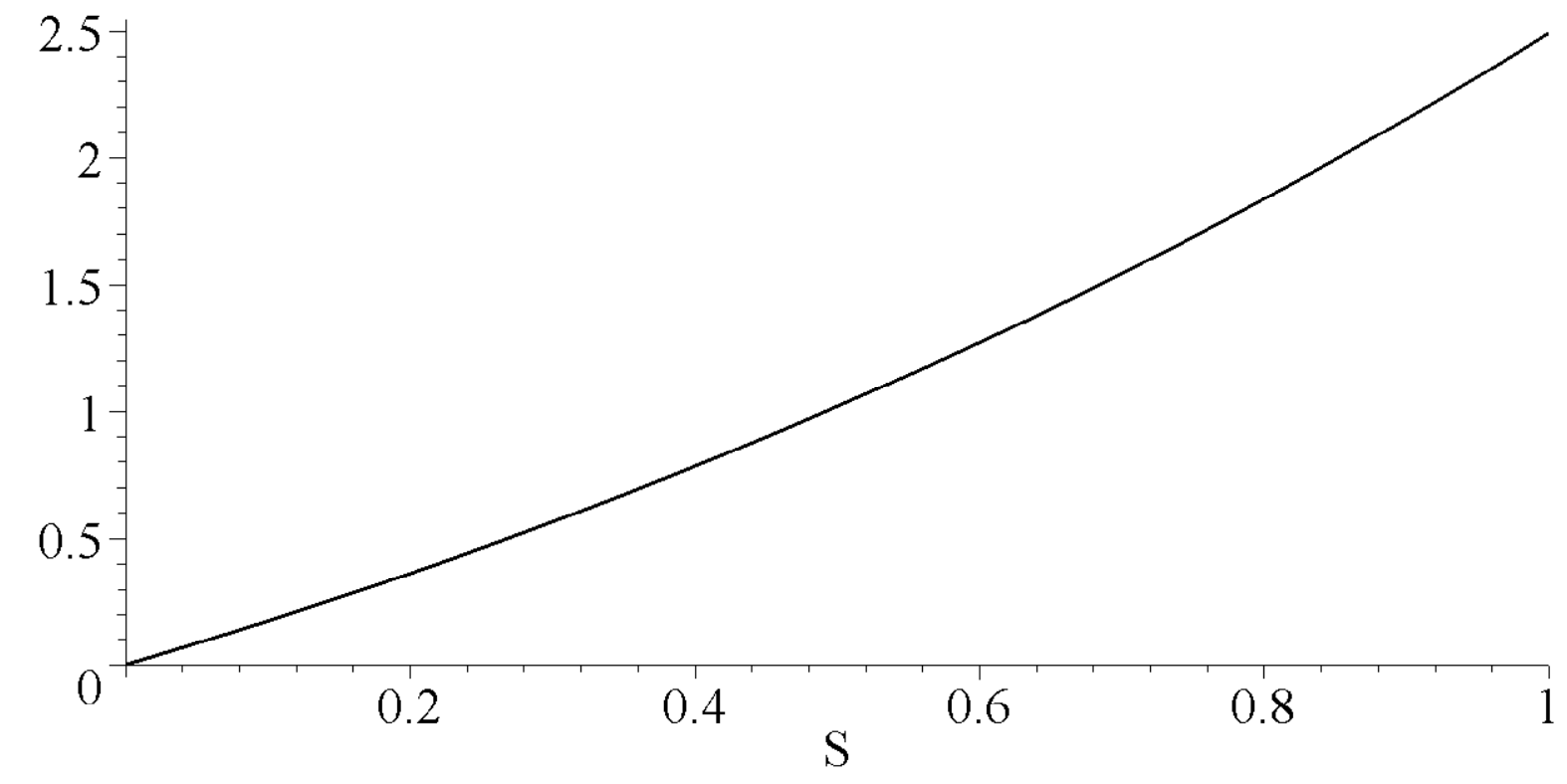}
\end{center}
\caption{A plot of $\frac{(a+g_u)(\left|I_u\right|^s-a^s)}{ba^s}$.}
\label{m3}
\end{figure}

We finish this section with a few such examples to show that the conditions of Theorem \ref{2GthmN} do in fact hold for a wide range of parameter values. Consider the $2$-vertex IFS (on the unit interval) of Figure \ref{P2Vertexb1} and let $a=\frac{11}{23}, \ g_u=\frac{5}{23}, \ b=\frac{7}{23}$ which we keep fixed. By varying the other three parameters $c$, $g_v$ and $d$ we may let the Hausdorff dimension $s$ range between $0$ and $1$. The graph of Figure \ref{m3} shows clearly that once condition (3) holds it continues to do so as $s$ increases. Putting $c=\frac{13}{73}$, $g_v=\frac{53}{73}$, $d=\frac{7}{73}$, the Hausdorff dimension is $s=0.4934118279$, and $\frac{h_v}{h_u}=0.5486642748 < 1$, so (2) holds, and (3) holds because $\frac{(a+g_u)(\left|I_u\right|^s-a^s)}{ba^s}=1.003400992 > 1$. Now increasing $c$ and $d$ will increase the Hausdorff dimension and so $(3)$ will continue to hold but eventually $(2)$ will fail. As an example let $c=\frac{43}{73},  g_v=\frac{7}{73}, d=\frac{23}{73}$ this gives a Hausdorff dimension of $s=0.7990855723$ but (2) fails because $\frac{h_v}{h_u}=1.152194154$.

Overall then this brief analysis does confirm that conditions (1), (2), and (3) will hold for a wide range of values of the parameters. Thus we have identified attractors of a class of $2$-vertex IFSs for which the Hausdorff measure is known.

\section{Gap lengths}\label{five}
In this section we only consider IFSs, $\bigl(V,E^*,i,t,r,((\mathbb{R},\left| \ \  \right|))_{v \in V},(S_e)_{e \in E^1}\bigr)$, for which the convex strong separation condition (CSSC) holds. The attractors of such IFSs can be written  as $(F_u)_{u \in V}= \bigcap_{k=0}^{\infty}(F_u^k)_{u \in V}$, where $F_u^k$  denotes the set of \emph{level-$k$ intervals at the vertex $u$}, see Subsection 2.2.1 \cite{phdthesis_Boore}, for the $1$-vertex case see \cite{Book_KJF2}. Some level-$k$ intervals are illustrated for a $2$-vertex IFS in Figure \ref{2G2VertexExample}. The CSSC ensures that if there are $n$ edges leaving a vertex $u$ then the level-$1$ intervals, $F_u^1$, will consist of $n$ disjoint intervals which will have $n-1$ open intervals between them. That is $I_u \setminus F_u^1 = \bigcup_{i=1}^{n-1}J_i^1$, where $I_u=C(F_u)$ and each $J_i^1$ is an open interval. The \emph{set of level-$1$ gap lengths at the vertex $u$} is defined as  
\begin{equation}
\label{G_u^1}
G_u^1=\biggl\{\left|J_i^1\right|    :  J_i^1 \textrm{ is an open interval in } I_u \setminus F_u^1 = \bigcup_{i=1}^{n-1}J_i^1 \biggr\}. 
\end{equation} 
In general $I_u \setminus F_u^k$ is a finite union of open intervals so $I_u \setminus F_u^k=\bigcup_{i \in H_k}J_i^k$, for some finite indexing set $H_k$. The \emph{set of level-$k$ gap lengths at the vertex $u$} is defined as
\begin{equation*}
G_u^k=\biggl\{\left|J_i^k\right| : J_i^k \textrm{ is an open interval in } I_u \setminus F_u^k = \bigcup_{i \in H_k}J_i^k \biggr\}. 
\end{equation*}
It follows that $I_u\setminus F_u= I_u \setminus \left( \bigcap_{k=0}^{\infty} F_u^k  \right)=\bigcup_{k=0}^{\infty} I_u\setminus F_u^k$. Since $I_u\setminus F_u^k=\bigcup_{i \in H_k}J_i^k$, for some finite indexing set $H_k$ and open intervals $J_i^k$ it is clear that $I_u\setminus F_u$ can be written as  a countable union of open intervals, $I_u\setminus F_u= \bigcup_{i=1}^{\infty}J_i$. We define the uniquely determined \emph{set of gap lengths of the attractor $F_u$} as
\begin{equation*}
\label{G_u definition}
G_u=\bigcup_{n=1}^{\infty}G_u^n = \biggl\{\left|J_i\right|    :  J_i \textrm{ is an open interval in } I_u \setminus F_u = \bigcup_{i=1}^{\infty}J_i \biggr\}. 
\end{equation*}

We now give an alternative description of the set $G_u$. For each edge $e \in E^1$ let $R_e:\mathbb{R} \to \mathbb{R}$ be the map $R_e(x)=r_ex$, where $r_e$ is the contracting similarity ratio of $S_e$. Let $\widetilde{f}:(K(\mathbb{R}^n))^{\#V} \to (K(\mathbb{R}^n))^{\#V}$, be the map defined by
\begin{equation*}
\widetilde{f}\left(\left(A_u\right)_{u \in V}\right) = \biggl(  \ \bigcup_{ e\in E_u^1  } R_e(A_{t(e)}) \ \bigcup \ G_u^1  \  \biggr)_{u \in V},  
\end{equation*}
for each $\left(A_u\right)_{u \in V} \in (K(\mathbb{R}^{n}))^{\#V}$. Here the sets of level-$1$ gap lengths, $(G_u^1)_{u \in V}$, which are called condensation sets in \cite{Book_Barnsley} for standard ($1$-vertex) IFSs, are clearly non-empty and compact so $(G_u^1)_{u \in V} \in (K(\mathbb{R}^{n}))^{\#V}$. It can be shown that $\widetilde{f}$ is a contraction on the complete metric space $((K(\mathbb{R}^{n}))^{\#V}, D_{\textup{H}})$, where $D_{\textup{H}}$ is the metric defined as the maximum of the coordinate Hausdorff metrics, see Theorem 9.1, \cite{Book_Barnsley}, for a proof for $1$-vertex IFSs. As 
\begin{equation}
\label{gap attractors}
\left(G_u\cup\left\{0\right\}\right)_{u \in V} = \biggl(  \ \bigcup_{ e\in E_u^1 } R_e(G_{t(e)}\cup\left\{0\right\}) \ \bigcup \ G_u^1  \  \biggr)_{u \in V},  
\end{equation}
the Contraction Mapping Theorem ensures that $\left(G_u\cup\left\{0\right\}\right)_{u \in V}$ is the unique fixed point of $\widetilde{f}$. The invariance Equations (\ref{Theorem 1 b}) and (\ref{gap attractors}) show clearly  the close relationship between attractors and their sets of gap lengths.

At a given vertex $u$ we can write the set of gap lengths in terms of similarity ratios of paths in the graph and level-$1$ gap lengths as
\begin{equation*}
G_u \ = \ \bigcup_{g_u\in G_u^1}g_u\bigl\{1,  r_{\be} : \be \in E_{uu}^*\bigr\} \quad \bigcup  \quad \bigcup_{\substack{ v\in V \\ v \neq u \\ g_v \in G_v^1}} g_v \bigl\{r_{\be} : \be \in E_{uv}^*\bigr\}.
\end{equation*}
  
For an IFS, $\bigl(V,E^*,i,t,r,((\mathbb{R},\left| \ \  \right|))_{v \in V},(S_e)_{e \in E^1}\bigr)$, for which the CSSC holds, Proposition 2.3.6 \cite{phdthesis_Boore} gives a constructive algorithm for calculating the set of gap lengths of any attractor as a finite union of cosets of finitely generated semigroups of positive real numbers. The generators of these semigroups are contracting similarity ratios of simple cycles in the directed graph. The algorithm works for any such IFS with no limit on the number of vertices in the directed graph.

We use the notation $(\mathbb{R}^+,\times)$ for the semigroup of positive real numbers under multiplication. For $x_i \in \mathbb{R}^+$, $1\leq i \leq j$, $\left\langle 1,x_1,x_2, \ldots ,x_j\right\rangle$ is the finitely generated subsemigroup (with identity) of $(\mathbb{R}^+,\times)$, where $\left\langle 1,x_1,x_2, \ldots ,x_j\right\rangle=\{x_1^{k_1}x_2^{k_2}\cdots x_j^{k_j}: k_i \in \mathbb{N}\cup \{ 0 \}, \,  1\leq i \leq j\}$ and for $y \in \mathbb{R}^+$ we write $y\left\langle 1,x_1,x_2, \ldots ,x_j\right\rangle$ for a coset with $y\left\langle 1,x_1,x_2, \ldots ,x_j\right\rangle=\{yx_1^{k_1}x_2^{k_2}\cdots x_j^{k_j}: k_i \in \mathbb{N}\cup \{ 0 \}, \, 1\leq i \leq j\}$. We will use the notation $\left\langle x_1,x_2, \ldots ,x_j\right\rangle_{\textrm{group}}=\{x_1^{k_1}x_2^{k_2}\cdots x_j^{k_j}: k_i \in \mathbb{Z},\, 1\leq i \leq j\}$ for the finitely generated group, the group operation again being multiplication.

Applying the algorithm of Proposition 2.3.6 \cite{phdthesis_Boore}, or alternatively by inspection, the gap lengths of the attractor $F_u$ of the $2$-vertex IFS (on the unit interval) of Figure \ref{P2Vertexb1} can be expressed as
\begin{align}
\label{gap2vertex}
G_u &= g_u \bigl\{1,  r_{\be} : \be \in E_{uu}^*\bigr\} \cup g_v \bigl\{r_{\be} : \be \in E_{uv}^*\bigr\} \nonumber \\
&= g_u\left\langle 1, a, bd\right\rangle \cup g_ubd\left\langle 1,  a, bd, c\right\rangle \cup g_vb\left\langle 1, a, bd, c\right\rangle  \nonumber \\
&=g_u\left\langle 1,a\right\rangle \cup g_ubd\left\langle 1,a,bd,c\right\rangle  \cup  g_vb\left\langle 1,a,bd,c\right\rangle. 
\end{align}
The generators of these semigroups are contracting similarity ratios of the simple cycles in the graph with $r_{e_1}=a$, $r_{e_2}r_{e_4}=bd$, and $r_{e_3}=c$.

Let $\bigl(V,E^*,i,t,r,((\mathbb{R},\left| \ \  \right|))_{v \in V},(S_e)_{e \in E^1}\bigr)$ be any $1$-vertex IFS, for which the CSSC holds, and which has $n$ edges leaving its single vertex then, as given in Equation (\ref{G_u^1}), the level-$1$ gap lengths are $G^1=\left\{g_j : 1 \leq j \leq n-1\right\}$ and the gap lengths of the attractor $F$ are given by
\begin{equation}
\label{gap1vertex}
G = \bigcup_{j=1}^{n-1}g_j\left\langle 1,  r_{e_1},  r_{e_2},  \ldots  , r_{e_n}\right\rangle,
\end{equation}
where $g_j, r_{e_i} \in \mathbb{R}^+$, and $r_{e_i}$, $1 \leq i \leq n$, are the contracting similarity ratios of the $n$ similarities $(S_e)_{e \in E^1}$. See Corollary 2.3.8 \cite{phdthesis_Boore}.

In the next section we use these expressions for gap lengths as a means of distinguishing between the attractors of $2$-vertex IFSs and the attractors of $1$-vertex IFSs for which the CSSC holds.
   
\section{Attractors of directed graph IFSs that are not attractors of standard IFSs for which the CSSC holds}\label{six}

We now simplify the $2$-vertex IFS (on the unit interval) of Figure \ref{P2Vertexb1} even further by taking the similarities $S_{e_2}$ and $S_{e_4}$ to have the same similarity ratio. That is we put $b=r_{e_2}=r_{e_4}=d$. The gap lengths of the attractor $F_u$ are now
\begin{equation}
\label{gap2vertexb}
G_u = g_u\left\langle 1,a\right\rangle \cup g_ub^2\left\langle 1,a,b^2,c\right\rangle  \cup  g_vb\left\langle 1,a,b^2,c\right\rangle, 
\end{equation}
by Equation (\ref{gap2vertex}). From Equations (\ref{gap2vertexb}) and (\ref{gap1vertex}), to prove that $F_u$ is an attractor which cannot be the attractor of any standard ($1$-vertex) IFS, for which the CSSC holds, it is enough to prove Lemma \ref{lemF}, which shows that $G_u$ cannot be the set of gap lengths of any $1$-vertex IFS for which the CSSC holds. We state this formally in Corollary \ref{corC}. We will need the following notion of multiplicative rational independence.

Let $U=\left\{u_1,  u_2,  \ldots  ,  u_r\right\}$ be a set of positive real numbers, then $U$ is a \emph{multiplicatively rationally independent} set if, for all integers $ m_i \in \mathbb{Z} $, $\sum_{i=1} ^r m_i \ln u_i = 0$, implies $m_i=0$ for all $i$, $1 \leq i \leq r$, or equivalently if $\prod_{i=1} ^r u_i^{ m_i } = 1$, then $m_i=0$ for all $i$, $1 \leq i \leq r$. 

\begin{lem}
\label{lemF}
Let $\left\{g_u, g_v, a, b,  c\right\}\subset \mathbb{R^+}$ be a multiplicatively rationally independent set. Then
\begin{equation*}
g_u\left\langle 1,a\right\rangle \cup g_ub^2\left\langle 1,a,b^2,c\right\rangle  \cup  g_vb\left\langle 1,a,b^2,c\right\rangle \neq \bigcup_{j=1}^{m}h_j\left\langle 1,  x_1,  x_2,  \ldots  , x_n\right\rangle,
\end{equation*}
for any  $h_j \in \mathbb{R}^+$, $1 \leq j \leq m$, and any $x_k \in \mathbb{R}^+$, $1 \leq k \leq n$. 
\end{lem}

\begin{proof}
For a contradiction we assume there exist positive real numbers $h_j$, $1 \leq j \leq m$, and $x_k$, $1 \leq k \leq n$, such that
\begin{equation}
\label{assumption}
g_u\left\langle 1,a\right\rangle \cup g_ub^2\left\langle 1,a,b^2,c\right\rangle  \cup  g_vb\left\langle 1,a,b^2,c\right\rangle = \bigcup_{j=1}^{m}h_j\left\langle 1,  x_1,  x_2,  \ldots  , x_n\right\rangle.
\end{equation}
This can be written as $g_uA \cup  g_vB = \bigcup_{j=1}^{m}h_j\left\langle 1,  x_1,  x_2,  \ldots  , x_n\right\rangle$, where
\begin{align*}
A  &= \left\langle 1,a\right\rangle \cup b^2\left\langle 1,a,b^2,c\right\rangle = \left\{a^pb^{2q}c^r  :  p,  q,  r  \in \mathbb{N}\cup\left\{0\right\}, \textrm{ if } q=0 \textrm{ then } r=0\right\}, \\
B &= b\left\langle 1,a,b^2,c\right\rangle = \left\{a^pb^{2q+1}c^r  :  p,  q,  r  \in \mathbb{N}\cup\left\{0\right\}\right\}.
\end{align*}
\noindent (a) $g_uA \cap g_vB = \emptyset$.

If $g_uA \cap g_vB \neq \emptyset$, then there exists $x \in g_uA \cap g_vB$ with $x=g_u^1g_v^0a^{p_1}b^{2q_1}c^{r_1}=g_u^0g_v^1a^{p_2}b^{2q_2+1}c^{r_2}$, which, by the rational independence of the set $\left\{g_u, g_v, a, b,  c\right\}$, implies $1=0$. This means that for any $h_j$, $1 \leq j \leq m$, either $h_j \in g_uA$ or $h_j \in g_vB$ but not both, so we consider each case in turn in parts (b) and (c).

\medskip

\noindent (b) \emph{$h_j \in g_uA$ implies $h_j\langle 1,  x_1,  x_2,  \ldots  , x_n\rangle \subset g_uA$ and $\langle 1,  x_1,  x_2,  \ldots  , x_n\rangle \subset \langle 1,a,b^2,c\rangle$}.

Suppose $h_j \in g_uA$ then $h_j = g_u^1a^{\alpha_1}b^{2\alpha_2}c^{\alpha_3}$. Let $x \in \left\langle  1,  x_1,  x_2,  \ldots  , x_n \right\rangle$. Assume $h_jx \in g_vB$ then $h_jx = g_u^1a^{\alpha_1}b^{2\alpha_2}c^{\alpha_3}x=g_v^1a^{\beta_1}b^{2\beta_2+1}c^{\beta_3}$, and  so $x$ is given by  $x=g_u^{-1}g_v^1a^{\beta_1-\alpha_1}b^{2(\beta_2-\alpha_2)+1}c^{\beta_3-\alpha_3}$, where we are now considering $x \in \langle g_u,g_v,a,b,c\rangle_{\textrm{group}}$. Consider any $k \in \mathbb{N}$ with $k \geq2$, then the exponent of $g_v$ in $h_jx^k$ is $k$. Either $h_jx^k \in g_uA$ or $h_jx^k \in g_vB$. If $h_jx^k \in g_uA$ then by rational independence $k=0$, which is a contradiction and if $h_jx^k \in g_vB$ then by rational independence $k=1$ which is again a contradiction. Therefore $h_jx \notin g_vB$ and so $h_jx \in g_uA$. 

We now write $h_jx$ as $h_jx = g_u^1a^{\alpha_1}b^{2\alpha_2}c^{\alpha_3}x=g_u^1a^{\beta_1}b^{2\beta_2}c^{\beta_3}$, with $x$ given as $x = a^{\beta_1-\alpha_1}b^{2(\beta_2-\alpha_2)}c^{\beta_3-\alpha_3}$, where strictly speaking we are again considering $x \in \langle g_u,g_v,a,b,c\rangle_{\textrm{group}}$. For any $k \in \mathbb{N}$, the exponent of $g_u$ in $h_jx^k$ is $1$ and so by rational independence, if $h_jx^k \in g_vB$, then $1=0$, which implies $h_jx^k \in g_uA$, and $h_jx^k = g_u^1a^{\alpha_1+k(\beta_1-\alpha_1)}b^{2\alpha_2+2k(\beta_2-\alpha_2)}c^{\alpha_3+k(\beta_3-\alpha_3)} = g_u^1a^{\delta_1(k)}b^{2\delta_2(k)}c^{\delta_3(k)}$, where $\delta_1(k), \delta_2(k)$,$\delta_3(k)$  $\in \mathbb{N}\cup \left\{0\right\}$. Again by rational independence we may conclude that $\alpha_1+k(\beta_1-\alpha_1)\geq 0$, $2\alpha_2+2k(\beta_2-\alpha_2)\geq 0$, and $\alpha_3+k(\beta_3-\alpha_3)\geq 0$, for all $k \in \mathbb{N}$. Hence $\beta_1-\alpha_1\geq 0$, $\beta_2-\alpha_2\geq 0$, $\beta_3-\alpha_3\geq 0$ so that $x \in \left\langle 1,a,b^2,c\right\rangle$. In summary we have shown that $h_j \in g_uA$ implies $h_jx \in g_uA$ for all $x \in \left\langle 1,  x_1,  x_2,  \ldots  , x_n\right\rangle$, and $\left\langle 1,  x_1,  x_2,  \ldots  , x_n\right\rangle \subset \left\langle 1,a,b^2,c\right\rangle$. 
  
\medskip

\noindent (c) \emph{$h_j \in g_vB$ implies $h_j\langle 1,  x_1,  x_2,  \ldots  , x_n\rangle \subset g_vB$ and $\langle 1,  x_1,  x_2,  \ldots  , x_n\rangle \subset \langle 1,a,b^2,c\rangle$}.

The proof is very similar to part (b), see Lemma 2.6.1 \cite{phdthesis_Boore}.

\medskip

Relabelling the $h_j$ if necessary, the results of parts (a), (b) and (c) imply that the set $\left\{h_1,  h_2, \ldots  ,  h_m\right\}$ must split into two non-empty subsets, $\left\{h_1,  h_2, \ldots  ,  h_r\right\}\subset g_uA$ and $\left\{h_{r+1},  h_{r+2}, \ldots  ,  h_m\right\}\subset g_vB$, with
\begin{align}
\label{g_uA}
g_uA &= \bigcup_{j=1}^{r}h_j\left\langle 1,  x_1,  x_2,  \ldots  , x_n\right\rangle, \\
g_vB &= \bigcup_{j=r+1}^{m}h_j\left\langle 1,  x_1,  x_2,  \ldots  , x_n\right\rangle, \nonumber
\end{align}
where $\left\langle 1,  x_1,  x_2,  \ldots  , x_n\right\rangle \subset \left\langle 1,a,b^2,c\right\rangle$.
 
\medskip

\noindent (d) \emph{At least one of the generators $x_k$, $1 \leq k \leq n$, of the semigroup $\left\langle 1,  x_1,  x_2,  \ldots  , x_n\right\rangle$, is of the form $x_k=c^t$, for some $t \in  \mathbb{N}$}. 

We recall that $g_uA = g_u\left\{a^pb^{2q}c^r  :  p,  q,  r,  \in \mathbb{N}\cup\left\{0\right\}, \textrm{ if } q=0 \textrm{ then } r=0\right\}$, so that $g_ub^2c^m \in g_uA$ for all $m \in \mathbb{N}$. Considering $M \in \mathbb{N}$ as fixed, then from Equation (\ref{g_uA})
\begin{equation}
\label{g_ub^2c^M}
g_ub^2c^M = h_sx_1^{i_1}x_2^{i_2}\cdots x_n^{i_n},
\end{equation}
for some $h_s \in g_uA$, $1 \leq s \leq r$, and non-negative integers $i_k \in  \mathbb{N}\cup \left\{0\right\}$, $1 \leq k \leq n$. For a contradiction we now assume that none of the $x_k$, $1 \leq k \leq n$, is of the form $x_k=c^t, \, t \in  \mathbb{N}$. Rational independence and the fact that $\left\langle 1,  x_1,  x_2,  \ldots  , x_n\right\rangle \subset \left\langle 1,a,b^2,c\right\rangle$, then implies that $h_s=g_uc^p$ and $x_k=b^2c^q$, for some $k, \,1 \leq k \leq n$, with $i_k=1$ and $i_l=0$ for all $l\neq k$, and where $p,q \in \mathbb{N}\cup\left\{0\right\}$, with $p+q=M$. That is Equation (\ref{g_ub^2c^M}) reduces to $g_ub^2c^M = h_sx_k$. Since we only have a finite number of generators in the semigroup $\left\langle 1,  x_1, x_2,  \ldots  , x_n\right\rangle$ and a finite set of numbers $\left\{h_i  : 1 \leq i \leq r \right\}$, we can only produce at most $r\times n$ distinct numbers of the form $g_ub^2c^m$, on the right-hand side of Equation (\ref{g_uA}). Therefore
\begin{equation*}
\left\{g_ub^2c^m  :  m \in \mathbb{N}\right\} \not\subset \bigcup_{j=1}^{r}h_j\left\langle 1,  x_1,  x_2,  \ldots  , x_n\right\rangle,
\end{equation*}
but
\begin{equation*}
\left\{g_ub^2c^m  :  m \in \mathbb{N}\right\} \subset g_uA.
\end{equation*}
This contradiction of Equation (\ref{g_uA}) means our assumption is false and at least one of the generators $x_k$, $1 \leq k \leq n$, must be of the form $x_k=c^t$, for some $t \in  \mathbb{N}$. 

\medskip

\noindent (e) $g_ua \notin \bigcup_{j=1}^{r}h_j\left\langle 1,  x_1,  x_2,  \ldots  , x_n\right\rangle$. 

From the result of part (d), relabelling the $x_k$ if necessary, so that $x_1=c^t$, $t \in \mathbb{N}$, we may write Equation (\ref{g_uA}) as
\begin{equation*}
g_uA   = g_u\left\langle 1,a\right\rangle \cup g_ub^2\left\langle 1,a,b^2,c\right\rangle = \bigcup_{j=1}^{r}h_j\left\langle 1,  c^t,  x_2,  \ldots  , x_n\right\rangle,
\end{equation*}
where $\left\langle 1,  c^t,  x_2,  \ldots  , x_n\right\rangle \subset \left\langle 1,a,b^2,c\right\rangle$. Now $g_u\left\langle 1,a\right\rangle \cap g_ub^2\left\langle 1,a,b^2,c\right\rangle= \emptyset$, by the rational independence of the set $\left\{g_u, g_v, a, b,  c\right\}$, so for each $j$, $1 \leq j \leq r$, either $h_j \in g_u\left\langle 1,a\right\rangle$ or $h_j \in g_ub^2\left\langle 1,a,b^2,c\right\rangle$ but not both.

Suppose $h_j \in g_u\left\langle 1,a\right\rangle$, then $h_j=g_ua^k$ for some $k \in \mathbb{N}\cup\left\{0\right\}$. It follows, again by rational independence, that $h_jc^t=g_ua^kc^t\notin g_u\left\langle 1,a\right\rangle$ and $h_jc^t=g_ua^kc^t\notin g_ub^2\left\langle 1,a,b^2,c\right\rangle$, that is $h_jc^t \notin g_uA$. This contradiction means $h_j \in g_ub^2\langle 1,a,b^2,c\rangle$ for each $j$, $1 \leq j \leq r$, and so we may write $h_j$ as $h_j=g_ub^2a^{k_j}b^{2l_j}c^{m_j}$, for some $k_j,l_j,m_j \in \mathbb{N}\cup\left\{0\right\}$. The rational independence of the set $\left\{g_u, g_v, a, b,  c\right\}$, together with the fact that $\left\langle 1,  c^t,  x_2,  \ldots  , x_n\right\rangle \subset \left\langle 1,a,b^2,c\right\rangle$, implies that
\begin{equation*}
g_ua \notin  \bigcup_{j=1}^{r}g_ub^2a^{k_j}b^{2l_j}c^{m_j}\left\langle 1,  c^t,  x_2,  \ldots  , x_n\right\rangle = \bigcup_{i=1}^{r}h_i\left\langle 1,  x_1,  x_2,  \ldots  , x_n\right\rangle.
\end{equation*} 

As $g_ua \in g_uA$, this is again a contradiction of Equation (\ref{g_uA}). Therefore our original assumption is false, that is Equation (\ref{assumption}) does not hold.    
\end{proof}

\begin{cor}
\label{corC}
For the $2$-vertex IFS (on the unit interval) of Figure \ref{P2Vertexb1}, but with $b=d$, if the set $\left\{g_u, g_v, a, b,  c\right\}\subset \mathbb{R^+}$ is a multiplicatively rationally independent set, then the attractor at the vertex $u$, $F_u$, is not the attractor of any standard ($1$-vertex) IFS, defined on $\mathbb{R}$, for which the CSSC holds.
\end{cor}

The next theorem  generalises Corollary \ref{corC} to a large class of directed graph IFSs, with any number of vertices, provided the directed graphs contain a particular subgraph. The proof is omitted but it is similar to the proof of Lemma \ref{lemF}, see Theorem 2.6.3 \cite{phdthesis_Boore}.

The \emph{vertex list} of a path $\be=e_1\cdots e_k \in E^*$ is  $v_1v_2v_3\cdots v_{k+1} = i(e_1) t(e_1) t(e_2) \cdots $ $t(e_k)$. A \emph{simple path} visits no vertex more than once, so a path $\be=e_1\cdots e_k\in E^*$ is simple if its vertex list contains exactly $k+1$ different vertices. A \emph{simple cycle} is a cycle which visits no vertex more than once apart from the initial and terminal vertices which are the same, so if $\be=e_1\cdots e_k\in E^*$ is a simple cycle then $i(\be)=t(\be)$ and its vertex list contains exactly $k$ different vertices. We say that \emph{two distinct paths are attached} if their vertex lists contain a common vertex or vertices. We also say that a \emph{path $\be$ is attached to a vertex $v$} if $v$ is in the vertex list of $\be$. A \emph{chain} is a finite sequence of distinct simple cycles where each simple cycle in the sequence is attached only to its immediate predecessor and successor cycles and to no other cycles in the sequence. A \emph{chain attached to a vertex $v$} is a chain of distinct simple cycles such that the first cycle in the sequence is attached to the vertex $v$ and thereafter no other cycle in the chain is attached to $v$.

\begin{thm}
\label{thmA}
Let $\bigl(V,E^*,i,t,r,((\mathbb{R},\left| \ \  \right|))_{v \in V},(S_e)_{e \in E^1}\bigr)$ be any directed graph IFS, satisfying the CSSC, whose directed graph contains three distinct simple cycles $\bc_1$, $\bc_2$, and $\bc_3$, such that $\bc_1$ is attached to a vertex $u$, $\bc_2\bc_3$ is a chain of length $2$ attached to $u$ and no chain in the graph, attached to $u$, contains both $\bc_1$ and $\bc_3$. Let $X_u \subset \mathbb{R}^+$, be the set of gap lengths and contracting similarity ratios
\begin{equation*}
X_u = \left\{g_w,  r_{\bc_i},  r_{\bp}  :  g_w \in G_w^1, \, w \in V, \, \bc_i \in T, \, \bp \in D_{uv}^*, \, v \in V, \, v \neq u\right\},
\end{equation*}
where $G_w^1$ is the set of level-$1$ gap lengths at the vertex $w \in V$, $T=\left\{\bc_i  :  i \in I\right\}$, the set of all simple cycles in the graph, and $D_{uv}^*\subset E_{uv}^*$, is the set of all simple paths from the vertex $u$ to the vertex $v$. 

Suppose the set $X_u$ is multiplicatively rationally independent, then the attractor at the vertex $u$, $F_u$, is not the attractor of any standard ($1$-vertex) IFS, defined on $\mathbb{R}$, for which the CSSC holds.
\end{thm}
 
We can take the simple cycles of Theorem \ref{thmA} to be $\bc_1=e_1$, $\bc_2=e_2e_4$ and $\bc_3=e_3$, for the edges $e_i$, $1 \leq i \leq4$, of the $2$-vertex IFS (on the unit interval) of Figure \ref{P2Vertexb1}. This means Theorem \ref{thmA} immediately yields the next corollary, with the set $X_u =\left\{  g_u, g_v, a, bd, c, b \right\}$. The set $X_u$ is multiplicatively rationally independent if and only if the set $\left\{g_u, g_v, a, b, c, d\right\}$ is multiplicatively rationally independent.

\begin{cor}
\label{corCb}
For the $2$-vertex IFS (on the unit interval) of Figure \ref{P2Vertexb1}, if the set $\left\{ g_u, g_v, a, b, c, d \right\}$ $\subset \mathbb{R^+}$ is a multiplicatively rationally independent set, then the attractor at the vertex $u$, $F_u$, is not the attractor of any standard ($1$-vertex) IFS, defined on $\mathbb{R}$, for which the CSSC holds.
\end{cor}

\section{Attractors of directed graph IFSs that are not attractors of standard IFSs}\label{seven}
Before proving Theorems \ref{2GthmU} and \ref{2GthmV} we first give some important consequences of $\mathcal{H}^s(F_u)=\left|I_u\right|^s$ in Lemmas \ref{2GlemO} and \ref{2GlemT}. The arguments we use are based on those employed by Feng and Wang in \cite{Paper_FengWang}. 

To illustrate the significance of Lemma \ref{2GlemO}, consider the $1$-vertex IFS defined on $\mathbb{R}$ by the similarities $S_1(x)=\frac{1}{3}x$, $S_2(x)=\frac{1}{27}x + \frac{4}{27}$, $S_3(x)=\frac{1}{3}x + \frac{2}{3}$. This is a modification of the Cantor set, C, which is the attractor of the IFS defined by $S_1$ and $S_3$, and for which $\mathcal{H}^s(C)=1$. The attractor $F$ is the unique non-empty compact set satisfying $F=\bigcup_{i=1}^3S_i(F)$. The OSC is satisfied for this IFS, this can be verified by taking the open set as $U=\left(0,\frac{1}{3}\right)\cup \left(\frac{2}{3},1\right)$. Actually the strong separation condition (SSC) holds but the CSSC does not. In particular, for $I=\left[0,1\right]$, $S_2(I)\subset S_1(I)$ but $S_1(F)\cap S_2(F)= \emptyset$. So $S_2(I)\subset S_1(I)$ does not imply $S_2(F)\subset S_1(F)$. In fact $S_1(F)\subsetneqq F \cap S_1(I)$, since $F \cap S_1(I)=S_1(F)\cup S_2(F)$. It follows by Lemma \ref{2GlemO}(b) that $\mathcal{H}^s(F)\neq \left|I\right|^s$, and so $\mathcal{H}^s(F)< 1$, by Lemmas \ref{2GlemD} and \ref{2GlemF}.  
\begin{lem}
\label{2GlemO}
Let $\bigl( V, E^{*}, i, t, r, (\mathbb{R},\left|  \ \ \right|)_{v \in V}, (S_e)_{e \in E^1} \bigr)$ be any directed graph IFS for which the OSC holds. For the attractor $F_u$ at the vertex $u$, let $s=\dimH F_u$ and $\left\{a_u,b_u\right\}\subset F_u \subset I_u=\left[a_u,b_u\right]$. Let $S_f:\mathbb{R} \to \mathbb{R}$ be any  similarity, with contracting similarity ratio $r_f$, $0<r_f<1$, and let $S_f(I_u)=\left[S_f(a_u),S_f(b_u)\right]=\left[a_f,b_f\right]$.

If $S_f(F_u)\subset F_u$ and $\mathcal{H}^s(F_u)=\left|I_u\right|^s$ then 
\begin{align*}
&\textup{(a)}  \quad \mathcal{H}^s(S_f(F_u))=\mathcal{H}^s(F_u\cap S_f(I_u))=\left(b_f-a_f\right)^s, \\
&\textup{(b)}  \quad S_f(F_u)= F_u\cap S_f(I_u). 
\end{align*}  
\end{lem}
\begin{proof}
(a) Clearly $S_f(F_u)\subset F_u\cap S_f(I_u)$, so
\begin{align*}
(b_f-a_f)^s&\geq \mathcal{H}^s(F_u\cap \left[a_f,b_f\right]) && (\textup{by Corollary \ref{2GcorG}(a)}) \\
&= \mathcal{H}^s(F_u\cap  S_f(I_u)) && \\
&\geq \mathcal{H}^s(S_f(F_u)) && \\
&= r_f^s\mathcal{H}^s(F_u) && (\textup{by the scaling property of the measure}) \\
&= \frac{(b_f-a_f)^s}{\left|I_u\right|^s}\mathcal{H}^s(F_u) && \\
&= (b_f-a_f)^s && (\textup{as $\mathcal{H}^s(F_u)=\left|I_u\right|^s$}). 
\end{align*}
(b) As $S_f(F_u)\subset F_u\cap S_f(I_u)$, we assume for a contradiction that $S_f(F_u)\subsetneqq F_u\cap S_f(I_u)$, so there exists a point $x\in F_u\cap S_f(I_u)$, such that $x \notin S_f(F_u)$. As $S_f(F_u)$ is compact, $\dist(x,S_f(F_u))>0$. The map, $\phi_u : E_u^\mathbb{N} \to F_u$, given by $\phi_u(\be) = x$,  $\left\{x\right\} = \bigcap_{k=1}^{\infty} S_{\be \vert_{k}}(F_{t(\be \vert_{ k})})$, is surjective so there exists an infinite path in the directed graph, $\be \in E_u^\mathbb{N}$, with $\left\{x\right\} = \bigcap_{k=1}^{\infty} S_{\be \vert_{ k}}(F_{t(\be \vert_{ k})})$. Now $\left(S_{\be \vert_{ k}}(F_{t(\be \vert_{ k})})\right)$, is a decreasing sequence of non-empty compact subsets of $F_u$, whose diameters tend to zero as $k$ tends to infinity, and so there exists $m \in \mathbb{N}$, such that $x \in S_{\be \vert_{m}}(F_{t(\be \vert_{m})})\subset F_u$, and $\left|S_{\be \vert_{m}}(F_{t(\be \vert_{m})})\right|<  \dist(x, S_f(F_u))$. It follows that $S_{\be \vert_{m}}(F_{t(\be \vert_{m})})\cap S_f(F_u)=\emptyset$. Also $x \in \left[a_f,b_f\right]$, and as $a_f, b_f \in S_f(F_u)$, $\dist(x, S_f(F_u))\leq x-a_f$ and $\dist(x, S_f(F_u))\leq b_f-x$ which means $S_{\be \vert_{m}}(F_{t(\be \vert_{m})})\subset \left[a_f,b_f\right]= S_f(I_u)$. In summary, 
\begin{equation}
\label{disjoint union}
( S_{\be \vert_{m}}(F_{t(\be \vert_{m})}) \cup S_f(F_u)) \subset (F_u\cap S_f(I_u)),
\end{equation}
where the union on the left hand side is disjoint.

By Theorem \ref{Theorem 2}, $\mathcal{H}^s(S_{\be \vert_{m}}(F_{t(\be \vert_{m})}))=r_{\be \vert_{m}}^s\mathcal{H}^s(F_{t(\be \vert_{m})})>0$, which we use to derive a contradiction as follows
\begin{align*}
(b_f-a_f)^s&= \mathcal{H}^s(F_u\cap S_f(I_u)) && (\textup{by part(a)}) \\
&\geq \mathcal{H}^s(S_{\be \vert_{m}}(F_{t(\be \vert_{m})}) \cup S_f(F_u)) && (\textup{by Equation (\ref{disjoint union})})\\
&= \mathcal{H}^s(S_{\be \vert_{m}}(F_{t(\be \vert_{m})})) + \mathcal{H}^s(S_f(F_u)) && (\textup{the union is disjoint}) \\
&> \mathcal{H}^s(S_f(F_u)) &&  \\
&= (b_f-a_f)^s && (\textup{by part (a)}).
\qedhere 
\end{align*} 
\end{proof}
The next inequality can be verified using calculus.
\begin{lem}
\label{2GlemS}
For $x,z>0$, $y\geq0$ and $0<p<1$,
\begin{equation*}
\left(x+y+z\right)^p < \left(x+y\right)^p + \left(y+z\right)^p - y^p.
\end{equation*}
\end{lem}
\begin{lem}
\label{2GlemT}
Let $\bigl( V, E^{*}, i, t, r, (\mathbb{R},\left|  \ \ \right|)_{v \in V}, (S_e)_{e \in E^1} \bigr)$ be any directed graph IFS for which the OSC holds. For the attractor $F_u$ at the vertex $u$, let $s=\dimH F_u$ and $\left\{a_u, b_u\right\}\subset F_u \subset I_u=\left[a_u, b_u\right]$. Let $S_f:\mathbb{R} \to \mathbb{R}$, $S_g:\mathbb{R} \to \mathbb{R}$, be any two distinct similarities with contracting similarity ratios $0<r_f,r_g<1$. 

If $S_f(F_u)\subset F_u$, $S_g(F_u)\subset F_u$, and $\mathcal{H}^s(F_u)=\left|I_u\right|^s$, then exactly one of the following three statements occurs,
\begin{align*}
&\textup{(a)}  \quad S_f(I_u)\cap S_g(I_u)=\emptyset, \textrm{ which implies } S_f(F_u)\cap S_g(F_u)=\emptyset, \\
&\textup{(b)}  \quad S_f(I_u)\subset S_g(I_u), \textrm{ which implies } S_f(F_u)\subset S_g(F_u), \\
&\textup{(c)}  \quad S_g(I_u)\subset S_f(I_u), \textrm{ which implies } S_g(F_u)\subset S_f(F_u).
\end{align*}  
\end{lem}
\begin{proof}This is similar to the claim in the proof of Theorem 4.1, \cite{Paper_FengWang}.

There are just five possibilities for the intervals $S_f(I_u)=\left[a_f,b_f\right]$, $S_g(I_u)=\left[a_g,b_g\right]$, 
\begin{align*}
\textup{(a)}  \quad \left[a_f,b_f\right]\cap &\left[a_g,b_g\right]=\emptyset, \quad \textup{(b)}  \quad \left[a_f,b_f\right]\subset \left[a_g,b_g\right], \quad \textup{(c)}  \quad \left[a_g,b_g\right]\subset \left[a_f,b_f\right],\\
&\textup{(d)}  \quad a_f < a_g\leq b_f < b_g, \quad \textup{(e)}  \quad a_g < a_f\leq b_g < b_f.
\end{align*}
First we prove that the situation in (d) cannot happen. 
\begin{align*}
 (b_g-a_f)^s&\geq \mathcal{H}^s(F_u\cap \left[a_f,b_g\right]) && (\textrm{by Corollary \ref{2GcorG}(a)}) \\
&= \mathcal{H}^s(F_u\cap \left[a_f,b_f\right]) + \mathcal{H}^s(F_u\cap \left[a_g,b_g\right]) && \\
& \quad \quad \quad \quad - \mathcal{H}^s(F_u\cap \left[a_g,b_f\right]) && (\textrm{a property of the measure})\\
&= \mathcal{H}^s(F_u\cap S_f(I_u)) + \mathcal{H}^s(F_u\cap S_g(I_u)) && \\
& \quad \quad \quad \quad - \mathcal{H}^s(F_u\cap \left[a_g,b_f\right]), && \\
&\geq (b_f-a_f)^s + (b_g-a_g)^s - (b_f-a_g)^s && (\textrm{by Lemma \ref{2GlemO}(a)} \\
& &&\quad \quad \textrm{and Corollary \ref{2GcorG}(a)}) \\
&> (b_g-a_f)^s. 
\end{align*}
The last inequality is obtained by putting $x=a_g-a_f>0$, $y=b_f-a_g\geq 0$, and $z=b_g-b_f>0$ in Lemma \ref{2GlemS}. This contradiction shows that (d) cannot occur and a similar argument can clearly be constructed to prove that (e) cannot happen either. Since $S_f \neq S_g$, exactly one of (a), (b), or (c) must occur. It only remains to prove the implications in the statement of the lemma.

That $S_f(I_u)\cap S_g(I_u)=\emptyset$ implies $S_f(F_u)\cap S_g(F_u)=\emptyset$ follows immediately as $S_f(F_u)\subset S_f(I_u)$ and $S_g(F_u)\subset S_g(I_u)$.

To see that $S_f(I_u)\subset S_g(I_u)$ implies $S_f(F_u)\subset S_g(F_u)$ we apply Lemma \ref{2GlemO}(b), to obtain $S_f(F_u)= F_u\cap S_f(I_u)\subset F_u\cap S_g(I_u)=S_g(F_u)$. 
 
Similarly, that $S_g(I_u)\subset S_f(I_u)$ implies $S_g(F_u)\subset S_f(F_u)$ also follows immediately by Lemma \ref{2GlemO}(b), since $S_g(F_u)= F_u\cap S_g(I_u)\subset F_u\cap S_f(I_u)=S_f(F_u)$. 
\end{proof}

We remind the reader that in the statements of Theorems \ref{2GthmU} and \ref{2GthmV} that follow, the $2$-vertex IFS (on the unit interval) of Figure \ref{P2Vertexb1} refers to Figure \ref{P2Vertexb1} but with $a_u$, $b_u$, $a_v$, and $b_v$ being assigned the specific values $a_u=a_v=0$ and $b_u=b_v=1$, so that $I_u=I_v=[0,1]$ and $\left|I_u\right|=\left|I_v\right|=1$, with the contracting similarity ratios and similarities as given in Equations (\ref{r_{e_i} unit interval}) and (\ref{S_{e_i} unit interval}).

\begin{thm}
\label{2GthmU}
For the $2$-vertex IFS (on the unit interval) of Figure \ref{P2Vertexb1}, suppose conditions (1), (2) and (3) of  Theorem \ref{2GthmN} all hold, so that $\mathcal{H}^s(F_u)=\left|I_u\right|^s=1$, and suppose also that the set $\left\{g_u,g_v,a,b,c,d\right\}\subset \mathbb{R^+}$ is multiplicatively rationally independent. 

Then the attractor at the vertex $u$, $F_u$, is not the attractor of any standard ($1$-vertex) IFS, defined on $\mathbb{R}$, with or without separation conditions.
\end{thm}
\begin{proof}
For a contradiction we suppose $F_u$ is the attractor of a $1$-vertex IFS, so $F_u$ will satisfy an invariance equation of the form
\begin{equation}
\label{onevertexequation}
F_u=\bigcup_{i=1}^n S_i(F_u),
\end{equation}
for some $n\geq2$. If $S_j(I_u)\cap S_k(I_u)\neq \emptyset$ for any $j\neq k$, $1 \leq j,k\leq n$, then by Lemma \ref{2GlemT}, either $S_j(I_u)\subset S_k(I_u)$, with $S_j(F_u)\subset S_k(F_u)$, or $S_k(I_u)\subset S_j(I_u)$, with $S_k(F_u)\subset S_j(F_u)$. Without loss of generality suppose $S_j(F_u)\subset S_k(F_u)$, then we may rewrite Equation (\ref{onevertexequation}) as 
\begin{equation*}
F_u=\bigcup_{ \substack{ i=1 \\ i\neq j}  } ^n S_i(F_u).
\end{equation*}
We may continue in this way, if necessary, relabelling and reducing the number of similarities $n$ in Equation (\ref{onevertexequation}) to $m$, $2 \leq m\leq n$, with
\begin{equation*}
F_u=\bigcup_{ i=1 } ^m S_i(F_u),
\end{equation*}
where $S_j(I_u)\cap S_k(I_u)= \emptyset$ for all $1 \leq j,k\leq m$, $j\neq k$. That is $F_u$ is the attractor of a $1$-vertex IFS that satisfies the CSSC. Because the set $\left\{g_u, g_v, a, b, c, d\right\}$ is multiplicatively rationally independent no such $1$-vertex IFS exists by Corollary \ref{corCb}. This is the required contradiction.
\end{proof}

\begin{thm}
\label{2GthmV}
For the $2$-vertex IFS (on the unit interval) of Figure \ref{P2Vertexb1}, but with $b=d$, suppose conditions (1), (2) and (3) of  Theorem \ref{2GthmN} all hold, so that $\mathcal{H}^s(F_u)=\left|I_u\right|^s=1$, and suppose also that the set $\left\{g_u,g_v,a,b,c\right\}\subset \mathbb{R^+}$ is multiplicatively rationally independent. 

Then the attractor at the vertex $u$, $F_u$, is not the attractor of any standard ($1$-vertex) IFS, defined on $\mathbb{R}$, with or without separation conditions. 
\end{thm}
\begin{proof} The proof is the same as that given for Theorem \ref{2GthmU}, except we apply Corollary \ref{corC} in place of Corollary \ref{corCb}.
\end{proof}

\begin{figure}[hbt]
\begin{center}
\includegraphics[trim = 5mm 212mm 5mm 16mm, clip, scale=0.7]{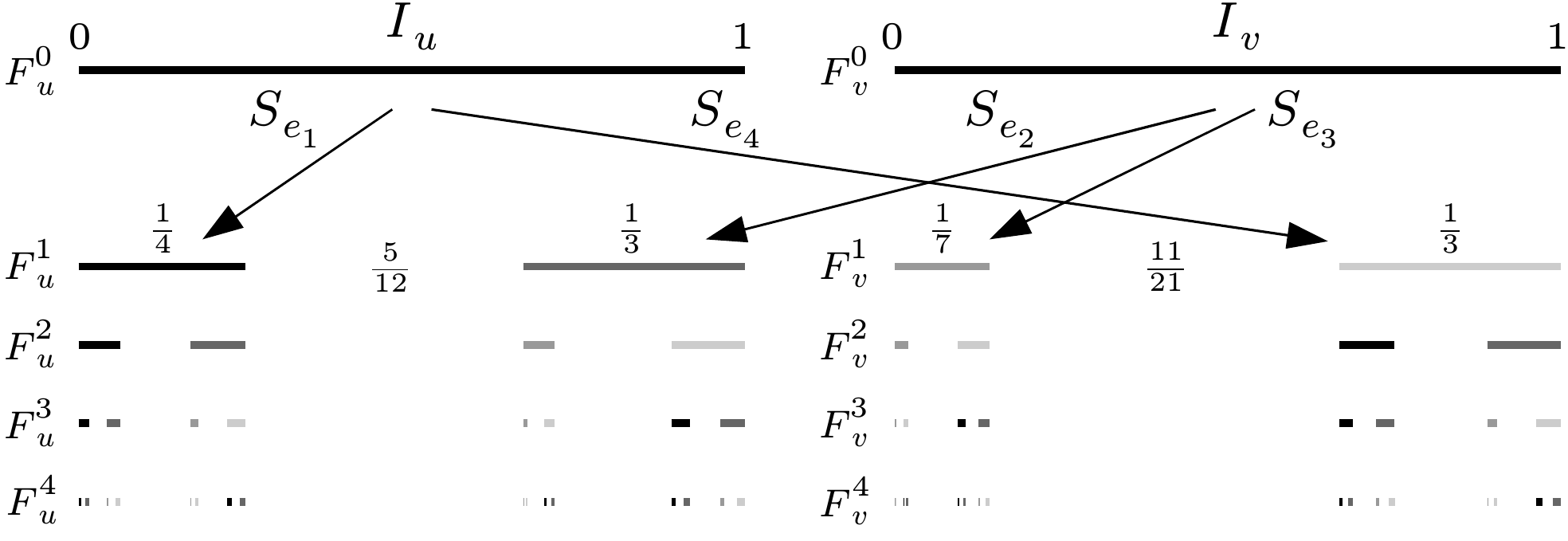}
\end{center}
\caption{Level-$k$ intervals for $0 \leq k \leq 4$. }
\label{2G2VertexExample}
\end{figure}

We now give a specific example to which we apply Theorem \ref{2GthmV}. Consider the following parameters for the $2$-vertex IFS (on the unit interval) of Figure \ref{P2Vertexb1}, with $a = \frac{1}{4}$, $g_u = \frac{5}{12}$, $b = d = \frac{1}{3}$, $c = \frac{1}{7}$, and $g_v = \frac{11}{21}$. The Hausdorff dimension can be calculated as $s=0.5147069928$, and $\frac{h_v}{h_u}=0.8978943038 < 1$. Also $\frac{(a+g_u)(\left|I_u\right|^s-a^s)}{ba^s}=2.082389923 > 1$, so conditions (1), (2) and (3) of Theorem \ref{2GthmN} all hold, which means $\mathcal{H}^s(F_u)=\left|I_u\right|^s=1$ and $\mathcal{H}^s(F_v)=0.8978943038$. 

The set $\left\{g_u, g_v, a, b, c\right\}=\left\{\frac{5}{12}, \frac{11}{21}, \frac{1}{4}, \frac{1}{3}, \frac{1}{7}\right\}$, is multiplicatively rationally independent. Theorem \ref{2GthmV} now ensures that the attractor $F_u$, at the vertex $u$, is not the attractor of any standard ($1$-vertex) IFS. Figure \ref{2G2VertexExample} illustrates the level-$k$ intervals, for $0 \leq k \leq 4$, for this particular example.

\bibliographystyle{amsplain}
\bibliography{P}

\providecommand{\bysame}{\leavevmode\hbox to3em{\hrulefill}\thinspace}
\providecommand{\MR}{\relax\ifhmode\unskip\space\fi MR }
\providecommand{\MRhref}[2]{%
  \href{http://www.ams.org/mathscinet-getitem?mr=#1}{#2}
}
\providecommand{\href}[2]{#2}
\begin{thebibliography}{10}

\bibitem{Paper_Ayer_Strichartz}
E.~Ayer and R.~S. Strichartz, \emph{Exact {H}ausdorff measure and intervals of
  maximum density for {C}antor sets}, Trans. Amer. Math. Soc. \textbf{351}
  (1999), 3725--3741.

\bibitem{Book_Barnsley}
M.~F. Barnsley, \emph{Fractals {E}verywhere}, Academic Press, San Diego, 1993.

\bibitem{phdthesis_Boore}
G.~C. Boore, \emph{Directed {G}raph {I}terated {F}unction {S}ystems}, Ph.D.
  thesis, School of Mathematics and Statistics, St Andrews University, July
  2011.

\bibitem{Paper_Delaware}
R.~Delaware, \emph{Every set of finite {H}ausdorff measure is a countable union
  of sets whose {H}ausdorff measure and content coincide}, Proc. Amer. Math.
  Soc. \textbf{131} (2002), 2537--2542.

\bibitem{Book_Edgar2}
G.~A. Edgar, \emph{Measure, {T}opology, and {F}ractal {G}eometry},
  Springer-Verlag, New York, 2000.

\bibitem{Paper_Edgar_Mauldin}
G.~A. Edgar and R.~D. Mauldin, \emph{Multifractal decompositions of digraph
  recursive fractals}, Proc. London Math. Soc. (3) \textbf{65} (1992),
  604--628.

\bibitem{Book_KJF1}
K.~J. Falconer, \emph{Techniques in {F}ractal {G}eometry}, John Wiley,
  Chichester, 1997.

\bibitem{Book_KJF2}
\bysame, \emph{Fractal {G}eometry, {M}athematical {F}oundations and
  {A}pplications}, John Wiley, Chichester, 2nd Ed. 2003.

\bibitem{Paper_FengWang}
De-J. Feng and Y.~Wang, \emph{On the structures of generating iterated function
  systems of {C}antor sets}, Adv. Math. \textbf{222} (2009), 1964--1981.

\bibitem{Paper_Hutchinson}
J.~Hutchinson, \emph{Fractals and self-similarity}, Indiana Univ. Math. J.
  \textbf{30} (1981), 713--747.

\bibitem{Paper_Marion}
J.~Marion, \emph{Mesure de {H}ausdorff d'un fractal $\grave{\textrm{a}}$
  similitude interne}, Ann. sc. Québec. \textbf{10} (1986), 51--84.

\bibitem{Paper_Mauldin_Williams}
R.~D. Mauldin and S.~C. Williams, \emph{{H}ausdorff dimension in graph directed
  constructions}, Trans. Amer. Math. Soc. \textbf{309} (1988), 811--829.

\bibitem{Book_Seneta}
E.~Seneta, \emph{Non-negative {M}atrices}, George Allen \& Unwin Ltd, London,
  1973.

\bibitem{Paper_Wang}
J.~Wang, \emph{The open set conditions for graph directed self-similar sets},
  Random Comput. Dynam. \textbf{5 4} (1997), 283--305.

\end{thebibliography}

\end{document}